\documentclass[11pt, DIV10,a4paper]{article}

\usepackage{natbib}
\usepackage{graphicx,subfig,amsmath,latexsym,amssymb}
\usepackage{float,epsfig,multirow,rotating,times}
\usepackage{upgreek,wrapfig}
\usepackage{comment}
\usepackage[colorlinks,citecolor=blue,urlcolor=blue,filecolor=blue,backref=page]{hyperref}
\usepackage{enumitem}
\usepackage{physics}
\usepackage{longtable}
\usepackage{pbox}

\newcommand {\ctn}{\citet} 
\newcommand {\ctp}{\citep}       
\newcommand{\mylabel}[2]{#2\def\@currentlabel{#2}\label{#1}}
\newcommand{\nmd}{\mathrm{NMD}}

\newcommand{\bd}{\boldsymbol{d}}

\newcommand{\btheta}{\boldsymbol{\theta}}

\newcommand{\bbeta}{\boldsymbol{\beta}}

\newcommand{\bPhi}{\boldsymbol{\Phi}}

\newcommand{\bTheta}{\boldsymbol{\Theta}}

\newcommand{\bSigma}{\boldsymbol{\Sigma}}

\newcommand{\bG}{\boldsymbol{G}}

\newcommand{\bX}{\boldsymbol{X}}

\newcommand{\bz}{\boldsymbol{z}}

\newcommand{\bzero}{\boldsymbol{0}}

\newcommand{\bone}{\boldsymbol{1}}

\newcommand{\postp}{P_{\btheta|\bX_n}}
\newcommand{\tpostp}{\tilde{P}_{\btheta|\bX_n}}
\newcommand{\pexp}{E_{\btheta|\bX_n}}

\newcommand{\bThetainf}{\mathbf{\Theta}^\infty} 
\newcommand{\mm}{\mathcal{M}}
\newcommand{\dnm}{\delta_{\mathcal{NM}}}
\newcommand{\mn}{\mathcal{MN}}
\newcommand{\fdrx}{FDR_{\bX_n}}
\newcommand{\fnrx}{FNR_{\bX_n}}
\newcommand{\mfdr}{mFDR_{\bX_n}}
\newcommand{\mfnr}{mFNR_{\bX_n}}
\newcommand{\mg}{\mathcal{G}_n}
\newcommand{\thetao}{\boldsymbol{\theta}_0}

\DeclareMathOperator*{\argmax}{argmax}
\DeclareMathOperator*{\ess}{ess}

\newtheorem{theorem}{Theorem}

\newtheorem{corollary}[theorem]{Corollary}

\newtheorem{definition}[theorem]{Definition}

\newtheorem{lemma}[theorem]{Lemma}

\newtheorem{remark}[theorem]{Remark}

\newenvironment{proof}[1][Proof]{\textbf{#1.} }{\ \rule{0.5em}{0.5em}}

\numberwithin{equation}{section}
\numberwithin{algo}{section}
\numberwithin{table}{section}
\numberwithin{figure}{section}

\usepackage{setspace}
\usepackage[mathscr]{euscript}
\usepackage[margin=1in]{geometry}
\usepackage{nameref,zref-xr}
\zxrsetup{toltxlabel}

\graphicspath{{img/}}
\begin{document}

\normalsize

\title{\vspace{-0.8in}
{\bf Asymptotic Theory of Dependent Bayesian Multiple Testing Procedures Under Possible Model Misspecification}}
\author{Noirrit Kiran Chandra and Sourabh Bhattacharya\thanks{
Noirrit Kiran Chandra is a postdoctoral researcher at Department of Statistical Science, Duke University, USA, and Sourabh Bhattacharya 
is an Associate Professor in Interdisciplinary Statistical Research Unit, Indian Statistical
Institute, 203, B. T. Road, Kolkata 700108.
Corresponding e-mail: noirritchandra@gmail.com.}}
\date{\vspace{-0.5in}}
\maketitle%

\begin{abstract}
We study asymptotic properties of Bayesian multiple testing procedures and provide sufficient conditions for strong consistency under general dependence structure. We also consider a novel Bayesian multiple testing procedure and associated error measures that coherently accounts for the dependence structure present in the model. We advocate posterior versions of
FDR and FNR as appropriate error rates and show that their asymptotic
convergence rates are directly associated with the Kullback-Leibler divergence
from the true model. Our results hold even when the class of postulated models is
misspecified. We illustrate our results in a variable selection problem with autoregressive response variables, and compare the new Bayesian procedure with some existing methods through extensive simulation studies in the variable selection problem. Superior performance of the new procedure compared to the others vindicate that proper exploitation of the dependence
structure by multiple testing methods is indeed important. Moreover, we obtain encouraging results in a real, maize data context, where we select influential marker variables.
\\[2mm]
{\bf MSC 2010 subject classifications:} Primary 62F05, 62F15; secondary 62C10, 62J07.
\\[2mm]
{\bf Keywords:} Bayesian multiple testing, Dependence, False discovery rate, Kullback-Leibler, Misspecified model, Posterior convergence.
\end{abstract}

\section{Introduction}
\label{sec:introduction}
In recent times there have been a tremendous growth in the area of multiple hypothesis testing as simultaneous inference on several parameters are often necessary. \ctn{Benjamini95} introduced a powerful approach to handle this problem in their landmark paper. However, in most real life situations the test statistics are generally dependent. \ctn{by01} showed that the Benjamini-Hochberg procedure is valid under positive dependence.
\ctn{berry99} have given a Bayesian perspective on multiple testing where the tests are related through a dependent prior. \ctn{scott10} discussed how empirical Bayes and fully Bayes methods adjust multiplicity. 

There are many works in the statistical literature on optimality and asymptotic behaviour of multiple testing methods in dependent cases. \ctn{sun2007} have proposed an optimal adaptive procedure where the data is generated from a two-component mixture model. \ctn{finner2002,efron2007} discussed the effects of dependence of error rates, among others. \ctn{finner2009} proposed new step-up and step-down procedures which asymptotically maximize power while controlling $FDR$. \ctn{caioptim11} have proposed an asymptotic optimal decision rule for short range dependent data with dependent test statistics. 

In this article, we study asymptotic properties of loss-function based Bayesian multiple testing procedures under general dependence setup. 
We show that under mild conditions such procedures are consistent in the sense that the decision rules converge to the truth with increasing sample size, even under dependence. 
We also show that the derived results hold even when the class of postulated models do not contain the true data generating process, that is, when the class of proposed models is misspecified.

\ctn{finner2007} discussed the effect of dependent test statistics on the \textit{false discovery rate} $(FDR)$.  \ctn{schwartzman2011} and \ctn{fan2012} discussed estimation of $FDR$ under correlation. 
In the frequentist multiple testing domain, the common practice is to control $FDR$ or the \textit{false non-discovery rate} $(FNR)$. Therefore, in that domain, asymptotic study of $FDR$ or $FNR$ in dependent cases has been done under different set ups. However, in the Bayesian literature, asymptotic study of the aforementioned error rates is not regular, although in practice, it is necessary to control those error rates. In this article, we conduct asymptotic analyses on these error rates under general dependent setup. We show that these error rates are directly associated to the Kullback-Leibler (KL) divergence from the true model in terms of their asymptotic convergence rates.

In the frequentist multiple testing setup, the decision rule for a hypothesis generally depends only on the corresponding test statistics. Bayesian loss-function based multiple testing methods are generally based on marginal posterior probabilities of a null hypothesis hypothesis being true or false. Most of the existing methods are marginal in the sense that the decision rule for a hypothesis do not depend on decisions of other hypotheses. 
Indeed, an important issue that seems to have received relatively less attention is that by proper utilization of the dependence structure among different hypotheses, 
the efficiency of multiple testing procedures can be significantly improved. \ctn{sun2009} have showed that incorporating the dependence structure of the parameters in the 
testing procedure increase efficiency. 

The aforementioned discussion points towards taking decisions regarding the hypotheses jointly.
In this regard, \ctn{chandra2017} developed a novel Bayesian multiple testing method which coherently takes the dependence structure among the hypotheses into consideration. 
In their method, the decisions are obtained jointly, as functions of appropriate joint posterior probabilities, and hence the method is referred to as a non-marginal Bayesian procedure. The procedure is based on new notions of error and non-error terms associated with breaking up the total number of hypotheses. 
They have shown that by virtue of the joint decision making principle, the non-marginal procedure has the desirable compound decision theoretic properties and for large samples, minimizes the KL divergence from the true data generating process, under general dependence models. Further, with extensive simulation studies they demonstrate significant gain in power over the existing marginal multiple testing methods, both classical and Bayesian. Application of this method to a deregulated microRNA discovery problem yielded insightful results which could not be obtained otherwise \ctp{chandra_mirna}.
In the following section we briefly describe the multiple testing procedure.

\subsection{A new non-marginal Bayesian multiple testing procedure}
Let $\bX_n=\{X_1,\ldots,X_n\}$ denote the available data set. Suppose the data is modelled by the family of distributions $P_{\bX_n|\btheta}$ (which may also be non-parametric). 
For $M>1$, let us denote by $\bTheta=\Theta_1\times\cdots\times\Theta_M$ 
the relevant parameter space associated with $\btheta=(\theta_1,\ldots,\theta_M)$, where we allow $M$ to be infinity as well. 
Let $\postp(\cdot)$ and $\pexp(\cdot)$ denote the posterior distribution and expectation respectively of $\btheta$ given $\bX_n$ and let $P_{\bX_n}(\cdot)$ and $E_{\bX_n} (\cdot)$ denote the marginal distribution and expectation of $\bX_n$ respectively. Let us consider the problem of testing $m$ hypotheses simultaneously corresponding to the actual parameters of interest, where
$1<m\leq M$. In this work, however, we assume $m$ to be finite. 

Without loss of generality, let us consider testing the parameters associated with $\Theta_i$; $i=1,\ldots,m$,
formalized as:
$$ H_{0i}:\theta_i \in \Theta_{0i}  \hbox{ versus } H_{1i} : \theta_i \in \Theta_{1i},$$ 
where $\Theta_{0i} \bigcap \Theta_{1i}=\emptyset \mbox{ and } \Theta_{0i} \bigcup \Theta_{1i} 
= \Theta_{i},\mbox{ for $i=1,\cdots,m$}.$

Let
\begin{align*}
d_i=&\begin{cases}
1&\text{if the $i$-th hypothesis is rejected;}\\
0&\text{otherwise;}
\end{cases}\\
r_i=&\begin{cases}
1&\text{if $H_{1i}$ is true;}\\
0&\text{if $H_{0i}$ is true.} 
\end{cases}
\end{align*}

In many real life situations, dependent prior structure is envisaged on the parameter space based on available domain knowledge. For example in spatial statistics, Gaussian process prior is often considered. In fMRI data, Gaussian Markov random field prior is a common prior. In such cases, the additional information on the parameters are incorporated in the model through the prior distribution. Various applications in recent times in fields as diverse as spatial statistics and environment \citep{Risser19}, 
time series \citep{Scott09}, neurosciences \citep{Brown14}, biological sciences \citep{Jensen09}, 
to name only a few,
consider Bayesian models with dependent prior structures. The basic idea behind the new multiple testing methodology is to incorporate such information, when available, in the testing procedure to obtain improved decision rule. This principle is in accordance with the traditional Bayesian philosophy
that when prior information is available, inference can be enhanced. 

Let $G_i$ be the set of hypotheses (including hypothesis $i$) where the parameters are 
dependent on $\theta_i$. In the new procedure, the decision of each hypothesis is penalized by incorrect decisions regarding other dependent parameters resulting in a compound criterion where all the decisions in $G_i$ deterministically depends upon each other. 
Define the following quantity
\begin{equation}
z_i=\begin{cases}
1&\mbox{if $H_{d_jj}$ is true for all $j\in G_i\setminus\{i\}$;}\\
0&\mbox{otherwise.}
\end{cases}\label{eq:z}
\end{equation}
If, for any $i\in\{1,\ldots,m\}$, $G_i=\{i\}$, a singleton, then we define $z_i=1$.
The notion of true positives $(TP)$ are modified as the following
\begin{equation}
TP=\sum_{i=1}^md_ir_iz_i,
\label{eq:tp}
\end{equation}
The posterior expectation of $TP$ is maximized subject to controlling the posterior expectation of the error term
\begin{equation}
E=\sum_{i=1}^md_i(1-r_iz_i).
\label{eq:e}
\end{equation}
It follows that the decision configuration can be obtained by minimizing the function
\begin{align}
\xi(\bd)&=-\sum_{i=1}^md_i\pexp(r_iz_i)+\lambda_n\sum_{i=1}^md_i \pexp (1-r_iz_i)\notag\\
&= -(1+\lambda_n)\sum_{i=1}^md_i\left(w_{in}(\bd)-\frac{\lambda_n}{1+\lambda_n}\right),\notag
\end{align}
with respect to all possible decision configurations of the form $\bd=\{d_1,\ldots,d_m\}$, where
$\lambda_n>0$,
and
\begin{equation*}
w_{in}(\bd)=\pexp(r_iz_i)= \postp\left(H_{1i}\cap\left\{\cap_{j\neq i,j\in G_i}H_{d_jj}\right\}\right),
\end{equation*}
is the posterior probability of the decision configuration $\{d_1,\ldots,d_{i-1},1,d_{i+1},\ldots,d_m\}$
being correct.
Letting $\beta_n=\lambda_n/(1+\lambda_n)$, one can equivalently maximize
\begin{equation}
f_{\beta_n}(\bd)=\sum_{i=1}^m d_i\left(w_{in}(\bd)-\beta_n\right)\label{eq:beta1}
\end{equation}
with respect to $\bd$ and obtain the optimal decision configuration.

\begin{definition}
	Let $\mathbb D$ be the set of all $m$-dimensional binary vectors denoting all possible decision configurations. Define $$\widehat{\bd}=\argmax_{\bd\in\mathbb{D}} f_\beta(\bd)$$ where $0<\beta<1$. Then $\widehat{\bd}$ is the \textit{optimal decision configuration} obtained as the solution of the non-marginal multiple testing method.
	\label{def:nmd}
\end{definition}

Note that in the definitions of both $TP$ and $E$, we penalize $d_i$ by incorrect decisions in the same group. Thus we design a compound criterion where decisions regarding dependent parameters deterministically depend upon each other adjudging other dependent parameters.


It is to be noted that there exist several cluster-based methods in the literature of multiple hypotheses testing. The works of \ctn{benjamini2007,sun15} are important to mention in this respect, among others. However, the $G_i$s in (\ref{eq:z}) are not to be confused with the notion of clusters in the aforementioned works. In their approaches a particular cluster of parameters is regarded as a signal or not. Essentially the decisions regarding the parameters in their clusters are same. However, that is not the case for our non-marginal method. 
The motivation behind our grouping is to borrow strength through the dependence structure across dependent parameters. This is a common practice in various applications \ctp{zhang2011,LIU2016}.

\subsection{Choice of $G_1,\ldots,G_m$}
\label{subsec:choiceofgroups}
Note that the non-marginal method depends on the choice of $G_i$s. However, in implementation
of the method, forming groups based on all dependent parameters might be disadvantageous in high
dimensional cases. Keeping very weakly dependent parameters in $G_i$ would increase the complexity of the method without
providing much extra information about the dependence structure. It would incur over-penalization levying high
posterior probability of $z_i = 0$. This might turn the method to
be overly conservative. Therefore, we recommend to restrict the group sizes proportional to the correlation among the parameters. \ctn{chandra2017} have prescribed the following strategy of group formation.

Let $\Lambda$ be the prior correlation matrix of $\btheta$. Let the $(i,j)$-th element of $\Lambda$ be $\lambda_{ij}$. We first consider the correlations between 
the $i$-th and $j$-th parameters, with $i<j$, 
and obtain a desired percentile $\lambda$ of these quantities. 
Then, in $G_i$ we include only those indices $j~(\neq i)$ such that $\lambda_{ij}\geq\lambda$.
Thus, the $i$-th group contains indices of the parameters that are highly correlated with the $i$-th parameter. If there exists no index $j$ such that $\lambda_{ij}\geq\lambda$, then $G_i=\{i\}$. 
In our applications we have considered $\lambda$ to be the 95-$th$ percentile, which is seen to have yielded good results.

Once the prior associated with the model is decided and well-chosen, the $G_i$s as defined above will also be fixed and would lead to reliable results. 
In case the prior information on the correlation structure of the parameters is weak, $\Lambda$ can be considered as the posterior correlation matrix of the parameters. Groups formed on the basis of the true correlation gives the best result as expected. However, groups formed on the basis of posterior correlation significantly improves the performance \ctp{chandra2017}. In Section \ref{simulation_asymptotic}, the groups are formed on the basis of posterior correlation and the strategy has outperformed some popular existing multiple testing methods in a variable selection context.

\label{pg:robust} Notably for large samples, Bayesian methods are usually robust with respect to prior choice and there is a huge literature formalizing this aspect. For example \ctn{schwartz1965,ghosal2000} discussed that Bayesian models are asymptotically consistent given that the priors satisfy certain regularity conditions. 
In the same vein we study the asymptotic properties of the Bayesian non-marginal method in this article and show that the procedure is asymptotically robust with respect to the 
choice of group structure later in Section \ref{subsec:asymp_robust}. 
In the same section we provide sufficient conditions for the asymptotic consistency of the non-marginal method. For illustrative purposes, we show that the conditions hold under a very general class of prior distributions in a time-varying covariate selection problem where the response variables possess inherent autocorrelation structure for any proper prior distribution over the parameter space.



\subsection{Existing and new error measures in multiple testing}
\label{subsec:Bayesian_errors}

\ctn{storey03} advocated \textit{positive False Discovery Rate} $(pFDR)$ as a measure of type-I error in multiple testing. Let $\delta_{\mm}(\bd|\bX_n)$ be the probability of choosing $\bd$ as the optimal decision configuration given data $\bX_n$ when a multiple testing method $\mm$ is employed. Then $pFDR$ is defined as:
\begin{equation*}
pFDR=E_{\bX_n} \left[ \sum_{\bd\in\mathbb{D}}  
\frac{\sum_{i=1}^{m}d_i(1-r_i)}{\sum_{i=1}^{m}d_i}\delta_\mm(\bd|\bX_n)\bigg{|}\delta_\mm(\bd=\mathbf{0}|\bX_n)=0 \right].
\end{equation*}

Analogous to type-II error, the \textit{positive False Non-discovery Rate} $(pFNR)$ is defined as
\begin{align*}
pFNR= E_{\bX_n}\left[\sum_{\bd\in\mathbb D} \frac{\sum_{i=1}^m(1-d_i)r_i} {\sum_{i=1}^m(1-d_i)} \delta_{\mathcal M}\left(\bd|\bX_n\right)
\bigg | \delta_{\mathcal M}\left(\bd=\bone|\bX_n\right)=0\right].
\end{align*}

Under prior $\pi(\cdot)$, \ctn{SanatGhosh08} defined posterior $FDR$ and $FNR$. The measures are given as following:
\begin{align*}
posterior~FDR
&= \pexp\left[\sum_{\bd\in\mathbb D}\frac{\sum_{i=1}^md_i(1-r_i)}{\sum_{i=1}^md_i \vee 1}\delta_{\mathcal M}\left(\bd|\bX_n\right) \right]= \sum_{\bd\in\mathbb{D}} 
\frac{\sum_{i=1}^{m}d_i(1-v_{in})}{\sum_{i=1}^{m}d_i \vee 1}\delta_{\mathcal M}(\bd|\bX_n); \\
posterior~FNR
&= \pexp\left[\sum_{\bd\in\mathbb D} \frac{\sum_{i=1}^m(1-d_i)r_i} {\sum_{i=1}^m(1-d_i)\vee 1} \delta_{\mathcal M}\left(\bd|\bX_n\right)
\right]=\sum_{\bd\in\mathbb D}
\frac{\sum_{i=1}^{m}(1-d_i)v_{in}}{\sum_{i=1}^{m}(1-d_i)\vee 1}\delta_{\mathcal M}(\bd|\bX_n),
\end{align*}

where $v_{in}=\postp(\varTheta_{1i})$. Also under any non-randomized decision rule $\mathcal{M}$, $\delta_{\mathcal{M}}(\bd|\bX_n)$ is either 1 or 0 depending on data $\bX_n$. Given $\bX_n$, we denote these posterior error measures by $\fdrx$ and $\fnrx$ respectively.

With respect to the new notions of errors in (\ref{eq:tp}) and (\ref{eq:e}),
$\fdrx$ is modified as
\begin{align*}
modified~\fdrx &=\pexp\left[ \sum_{\bd\in\mathbb D}\frac{\sum_{i=1}^md_i(1-r_iz_i)}{\sum_{i=1}^md_i\vee1}
\delta_{\mathcal M}\left(\bd|\bX_n\right)\right] 
\notag\\
&=  \sum_{\bd\in\mathbb{D}}  
\frac{\sum_{i=1}^{m}d_i(1-w_{in} (\bd))}{\sum_{i=1}^{m}d_i\vee1}\delta_{\mathcal M} (\bd|\bX_n).
\end{align*}

We denote $modified~\fdrx$ and $\fnrx$ by $\mfdr$ and $\mfnr$ respectively. Notably, the expectations of $\fdrx$ and $\fnrx$ with respect to $\bX_n$, conditioned on the fact that their respective denominators are positive, yields the \textit{positive Bayesian} $FDR~(pBFDR)$ and $FNR$ $(pBFNR)$ respectively. The same expectation over $\mfdr$ yields \textit{modified positive} $BFDR~(mpBFDR)$.





We advocate the posterior error measures $\mfdr,~\fdrx$ and $\fnrx$ as multiple testing error controlling measures in Bayesian multiple testing. These measures give the performance of the employed multiple testing procedure given the data, and hence most appropriate from the Bayesian perspective. In particular, wisdom gained from the traditional debate between 
the classical and Bayesian paradigms suggests that avoiding expectation with respect to the data in the error measures can help avoid possible paradoxes analogous to examples such as
the Welch's paradox \ctp{Welch39}.  
Not only that, the posterior error measures are readily estimable in practical situations, however complicated the dependent structure may be, without any assumption. In Section \ref{sec:compare_BFDR}, we show that the asymptotic convergence rates of these measures are associated with the KL divergence between the true data generating process and the selected model. 
As regards $\mfdr$, it takes into account the joint dependence structure between parameters through the $z_i$ terms. As will be shown subsequently, this joint dependence
structure manifests itself through the convergence rate of $\mfdr$. 
\ctn{chandra2017} also showed that $\mfdr$ can be interpreted as the posterior probability of an incorrect decisions within each group. Extensive simulation studies indicated that controlling the $\mfdr$ gives better protection against the type-II error in dependent cases.


\ctn{muller04} considered the following additive loss function
\begin{equation}
L(\bd,\btheta)= c\sum_{i=1}^m d_i(1-r_i)+ \sum_{i=1}^m (1-d_i)r_i,
\label{eq:loss_mul}
\end{equation}
where $c$ is a positive constant.
The decision rule that minimizes the posterior risk of the above loss is 
$d_i=I\left(v_i>\frac{c}{1+c}\right)$ for all $i=1,\cdots,m$, where $I(\cdot)$ is the indication function. 

This loss function has been widely used in the Bayesian multiple testing setups and also in frequentist decision theoretic approaches \ctp{sun2009, caioptim11}. Notably, the non-marginal method boils down to this additive loss function based approach when $G_i=\{i\}$, that is, when the information regarding  dependence between hypotheses is not available or overlooked. Hence, the convergence properties of the additive loss function based methods can be easily derived from our theories. We discuss this subsequently later in this article.

It is to be seen that multiple testing problems can be regarded as model selection problems where the task is to choose the correct specification for the parameters under consideration. Even if one decision is taken incorrectly, the model gets misspecified. \ctn{Shalizi09} considered asymptotic behaviour of misspecified models under very general conditions. We adopt his basic assumptions and some of his convergence results to build a general asymptotic theory for our multiple testing method. 

In Section \ref{sec:compare_BFDR}, we provide the setup, assumptions and the main result which we adopt for our purpose. In the same section 
we investigate consistency of the non-marginal multiple testing procedure. In Section \ref{subsec:asymp_mpBFDR}, we study the rates of convergence of different versions of $FDR$s
and asymptotic comparison between them. 
In Section \ref{sec:compare_BFNR}, we investigate the asymptotic properties of different versions of $FNR$s. 
We then investigate, in Section \ref{sec:asymp_category_a}, the asymptotic properties of the relevant versions of $FNR$ when the multiple testing
methods are adjusted so that $mpBFDR$ and $pBFDR$ tend to $\alpha$, for some $\alpha\in(0,a)$, where $a\leq 1$.
Indeed, as we show, any value of $\alpha\in (0,1)$ is not permissible asymptotically.
We further show that the versions of $FNR$ tend to zero at a faster rate compared to the situations where $\alpha$-control is
not exercised.
In Section \ref{sec:ar1} we illustrate the asymptotic properties of the non-marginal method in a time-varying covariate selection problem where the response variables possess inherent autocorrelation structure. In Section \ref{simulation_asymptotic}, we compare the performance of this method with some popular existing Bayesian multiple testing methods.
In Section \ref{sec:realdata} we apply our non-marginal method to a variable selection problem in a real, maize data with 7389 covariates representing
SNP (single nucleotide polymorphism) markers, concerning linear regression
of ``days to anthesis male flowering time" on the covariates. Excellent fit is the outcome, once the significant variables have been been selected by our Bayesian 
multiple testing method.  
Finally, in Section \ref{sec:conclusion} we summarize our contributions and provide concluding remarks.


\section{Consistency of the non-marginal procedure and other procedures based on additive loss}
\label{sec:compare_BFDR}
\subsection{Preliminaries for ensuring posterior convergence under general setup}
\label{sec:shalizi}
We consider a probability space $(\Omega,\mathcal F, P)$, 
and a sequence of random variables $X_1,X_2,\ldots$,   
taking values in some measurable space $(\Xi,\mathcal X)$, whose
infinite-dimensional distribution is $P$. The natural filtration of this process is
$\sigma(\bX_n)$. We denote the distributions of processes adapted to $\sigma(\bX_n)$ 
by $P_{\bX_n|\btheta}$, where $\btheta$ is associated with a measurable
space $(\bTheta,\mathcal T)$, and is generally infinite-dimensional. 
For the sake of convenience, we assume, as defined by \ctn{Shalizi09}, that $P$
and all the $P_{\bX_n|\btheta}$ are dominated by a common reference measure, with respective
densities $p$ and $f_{\btheta}$. The usual assumptions that $P\in\bTheta$ or even $P$ lies in the support 
of the prior on $\bTheta$, are not required for Shalizi's result, rendering it very general indeed. We put the prior distribution $\pi(\cdot)$ on the parameter space $\bTheta$. Following Shalizi we first define some notations:
Consider the following likelihood ratio:
\begin{equation*}
R_n(\btheta)=\frac{f_{\btheta}(\bX_n)}{p(\bX_n)}.
\end{equation*}
For every $\btheta\in\Theta$, the KL-divergence rate $h(\btheta)$ is defined as 
\begin{equation}
h(\btheta)=\underset{n\rightarrow\infty}{\lim}~\frac{1}{n}E\left(\log\frac{p(\bX_n)}{f_{\btheta}(\bX_n)}\right),
\label{eq:KL_rate}
\end{equation}
given that the above limit exists. For $A\subseteq\bTheta$, let
\begin{align}
h\left(A\right)=\underset{\btheta\in A}{\mbox{ess~inf}}~h(\btheta);~ 
J(\btheta)=h(\btheta)-h(\Theta);~
J(A)=\underset{\btheta\in A}{\mbox{ess~inf}}~J(\btheta).\label{eq:J2}
\end{align}

We have stated the assumptions \ref{shalizi1}-\ref{shalizi7} considered by Shalizi in Section \ref{subsec:assumptions_shalizi} of the attached supplementary file. Under those assumptions the following theorem can be seen to hold:
\begin{theorem}[\ctp{Shalizi09}]
	\label{th:shalizi}
	Consider assumptions \ref{shalizi1}--\ref{shalizi7} and any set $A\in\mathcal T$ with $\pi(A)>0$. If $\varsigma>2h(A)$, where
	$\varsigma$ is given in (\ref{eq:S5_1}) under assumption \ref{s5}, then
	\begin{equation*}
	\underset{n\rightarrow\infty}{\lim}~\frac{1}{n}\log\postp(A|\bX_n)=-J(A).
	\end{equation*}
\end{theorem}
We shall frequently make use of this theorem for our purpose. Also throughout this article, we show consistency results for general models satisfying \ref{shalizi1}--\ref{shalizi7}. For all our results, we assume these conditions. 
\subsection{Some requisite notations for the non-marginal method}
\label{subsec:notations}

It is very interesting that we need not assume that the true data generating process $P$ is in the class of postulated model $F_{\btheta};~\btheta\in\bTheta$. However, asymptotic consistency of the non-marginal procedure can still be achieved in the sense that, with increasing sample size the model with minimal misspecification is selected.  Note that depending on $d_j=0$ or 1, $\Theta_{d_jj}$ is the specification corresponding to $\theta_j$ directed by $d_j$. Now for all possible decision configurations the parameter space $\bTheta$ can have the following partition
\begin{equation*}
\bTheta(\bd)=\prod_{i=1}^m \Theta_{d_ii}\times\prod_{i=m+1}^M \Theta_{i}.
\end{equation*}
Note that $J(\bTheta)$ is the minimal KL-divergence between the true data generating process $P$ and the class of all postulated models. Among all possible decision configurations $\bd\in \mathbb{D}$, let $\bd^t$ be such that $J(\bTheta(\bd^t))=J(\bTheta)$. Note that $\bd^t$ minimizes the
KL-divergence between the true data generating model $P$ among all possible decision configurations. We regard $\bd^t$ as the true decision configuration. We will show that with increasing sample size the non-marginal procedure will choose $\bd^t$ as the optimal decision rule almost surely $(a.s.)$. We now define some notations required for further advancements.
\[
\bTheta_{i,\bd}=\left\lbrace \theta_i\in\Theta_{1i}, \theta_j\in\Theta_{d_jj}~\forall~j\neq i~\&~j\in G_i \right\rbrace.
\]
Then $\bTheta_{i,\bd}$ is the joint parameter space for the parameters in $G_i$ directed by $\bd$. For any decision configuration $\bd$ and group $G$ let $\bd_G=\{d_j:j\in G\}$. Define
$$\mathbb D_{i}=\left\{\bd:~\mbox{all decisions in}~\bd_{G_i}~\mbox{are correct}\right\}.$$
Here $\mathbb D_i$ is the set of all decision configurations where the decisions corresponding to the hypotheses in $G_i$ are at least correct. Clearly $\mathbb D_i$ contains $\bd^t$ for all $i$.

Hence, $\mathbb D^c_{i}=\left\{\bd:~\mbox{at least one decision in} ~\bd_{G_i}~ \mbox{is incorrect}\right\}$.
%
By Theorem \ref{th:shalizi}, for any $\epsilon>0$, there exists $n_0(\epsilon)$ such that for each $i=1,\ldots,m$,
for $n\geq n_0(\epsilon)$,
\begin{align}
&\exp\left[-n\left(J\left(\bTheta_{i,\bd}\right)+\epsilon\right)\right]
<w_{in}(\bd)<\exp\left[-n\left(J\left(\bTheta_{i,\bd}\right)-\epsilon\right)\right]
~\mbox{if}~\bd\in\mathbb D^c_{i}, \label{eq:shalizi1}\\
&\exp\left[ -n(J(\bTheta_{i,\bd^t}^c)+\epsilon) \right] <1- w_{in}(\bd)<\exp\left[ -n(J(\bTheta_{i,\bd^t}^c)-\epsilon ) \right] ~\mbox{if}~\bd\in\mathbb D_i. \label{eq:shalizi2}
\end{align}

Also, for $i=1,\ldots,m$, and for $n\geq n_0(\epsilon)$,
\begin{align}
&\exp\left[-n\left(J\left(H_{1i}\right)+\epsilon\right)\right]
<v_{in}<\exp\left[-n\left(J\left(H_{1i}\right)-\epsilon\right)\right],~
\mbox{if}~d^t_i=0;
\label{eq:shalizi1_muller}\\
&1-\exp\left[-n\left(J\left(H_{0i}\right)-\epsilon\right)\right]<v_{in}
<1-\exp\left[-n\left(J\left(H_{0i}\right)+\epsilon\right)\right]~\mbox{if}~d^t_i=1\label{eq:shalizi2_muller}
\end{align}
where
$
J(\bTheta_{i,\bd})=\underset{\btheta\in\varPsi_{i\bd}}{\ess\inf} J(\btheta) ;
J(H_{ki})=\underset{\btheta\in\varUpsilon_{ki}}{\ess\inf} J(\btheta)
$
and 
\begin{align*}
\varPsi_{i\bd}=&\{\theta_i\in\Theta_{1i}, \theta_j\in\Theta_{d_jj}~\forall~j\neq i~\&~j\in G_i,\theta_k\in\Theta_k ~\forall~k\in G_i^c \}\mbox{ and}\\
\varUpsilon_{ki}=&\{\theta_i\in\Theta_{ki}, \theta_j\in\Theta_{j}~\forall~j\neq i \},~k=0,1.
\end{align*}

Note that in $\varPsi_{i\bd}$, $\theta_k$ lies in its whole parameter space for all $k\in G_i^c$, irrespective of the fact that $d_k$ might be incorrect. Hence, it corresponds to a model 
where only $\left\{\theta_k:k\in G_i\right\}$ may be misspecified. $J(\bTheta_{i,\bd})$ gives the KL-divergence rate (defined in \eqref{eq:J2}) between the true model and this model. 
Also 
$J\left(\Theta_{i\bd}\right)>0$ if $\bd\in\mathbb D^c_i$,  
$J\left(H_{1i}\right)>0$ if $d^t_i=0$ and  $J\left(H_{0i}\right)>0$ if $d^t_i=1$.

It is important to observe that, in the above equations \eqref{eq:shalizi1}-\eqref{eq:shalizi2_muller}, 
we have referred to the same $\epsilon$ and the same $n_0(\epsilon)$ for every $i=1,\ldots,m$. Due to the finiteness of $m$, taking the same $n_0(\epsilon)$ is possible here. 



\subsection{The basic consistency theory for multiple testing with application to the non-marginal and additive loss based procedures}
\label{subsec:convergence}
With the above notations, in this section we show that the non-marginal procedure is asymptotically consistent under any general dependent model satisfying the conditions 
in Section \ref{subsec:assumptions_shalizi}. 
As a simple corollary, we show that other existing multiple testing procedures based on additive loss, are also consistent. 
Let us first formally define what we mean by asymptotic consistency of a multiple testing procedure.
\begin{definition}
	Let $\bd^t$ be the true decision configuration among all possible decision configurations as defined in Section \ref{subsec:notations}. Then a multiple testing method $\mm$ is said to be asymptotically consistent if almost surely
	\begin{equation*}
	\lim_{n\rightarrow\infty} \delta_\mm(\bd^t|\bX_n) =1.
	\end{equation*}
	\label{def:consistent}
\end{definition}
We now state the requisite conditions for $\nmd$ to be asymptotically consistent.

\begin{enumerate}[label={(A\arabic*)}]	
	\item \label{A1} We assume that the sequence $\beta_n$ is neither too small nor too large, that is,
	\begin{align}
	\underline\beta&=\underset{n\geq 1}{\liminf}\beta_n>0;\label{eq:liminf_beta}\\
	\overline\beta&=\underset{n\geq 1}{\limsup}\beta_n<1.\label{eq:limsup_beta}
	\end{align}
	\item \label{A2} We assume that neither all the null hypotheses are true and nor all of then are false, that is, $\bd^t\neq\bzero$ and $\bd^t\neq\bone$, where
	$\bzero$ and $\bone$ are vectors of 0's and 1's respectively.
\end{enumerate}

Recall the constant $\beta_n$ in (\ref{eq:beta1}), which is the penalizing constant between the error $E$ and true positives $TP$ and \ref{A1} ensures a fine balance between these two. It is necessary for the asymptotic consistency of both the non-marginal method and additive loss function based method. Notably, \ref{A2} is not required for the consistency results. Its role is to ensure that the denominator terms in the multiple testing error measures (defined in Section \ref{subsec:Bayesian_errors}) do not become 0 by ruling out two very extreme situations where none/all of the null hypotheses are false. It is also important in the asymptotic studies of the different versions of $FDR$ and $FNR$ that we consider. With these conditions we propose and prove the following results.
\begin{theorem}
	\label{th:nmd_consistent}
	Let $\dnm(\cdot|\bX_n)$ denote the non-marginal decision rule given data $\bX_n$. 
	Assume condition \ref{A1} on $\beta_n$. Then the non-marginal decision procedure is asymptotically consistent.
\end{theorem}

\begin{remark}
	\label{remark}
	It is important to note that in the proof of Theorem \ref{th:nmd_consistent}, we do not require any assumption on how the groups should be formed. The theorem is valid even if the groups are implicitly dependent on the observed data. This shows that, in case the prior information on the correlation structure of the parameters is weak, the non-marginal method is also valid when the groups are formed on the basis of posterior correlation or by other data-adaptive methods.
\end{remark}

We have already mentioned that the optimal decision rules corresponding to the loss function in (\ref{eq:loss_mul}) is a special case of the non-marginal method when dependence among the hypotheses is ignored. As we have not considered any particular structure of $G_i$'s in Theorem \ref{th:nmd_consistent}, consistency of the additive loss-function based method can also be obtained from the previous theorem.
\begin{corollary}
	Assuming condition \ref{A1}, the optimal decision rule corresponding to the additive loss function (\ref{eq:loss_mul}) is asymptotically consistent.
\end{corollary}


\subsection{Asymptotic robustness with respect to group choice}
\label{subsec:asymp_robust}
Note that in Theorem \ref{th:nmd_consistent} no specification on group formation is required. Generally for large samples, Bayesian methods are robust with respect to the prior choice given that the prior distribution follows some regularity conditions. Theorem \ref{th:nmd_consistent} entails that the non-marginal method is consistent for any group choice given that the model and prior distributions satisfy the conditions of Section \ref{subsec:assumptions_shalizi}. Hence, we see that the non-marginal method is asymptotically robust with respect to the choice of groups. 

\section{Asymptotic analyses of multiple testing error rates}
\label{subsec:asymp_mpBFDR}
\subsection{Asymptotic properties of versions of $FDR$}
First we study the convergence properties of $\mfdr$ and $\fdrx$ in this section. 
We show that the convergence rates of the posterior error measures are directly associated with the KL-divergence from the true model.

\begin{theorem}
	\label{theorem:preferability}
	Assume conditions \ref{A1}-\ref{A2}. Let $J_{\min}=\underset{i:d_i^t=1}{\min} J(\bTheta_{i,\bd^t}^c)$ and $H_{\min}=\underset{i:d_i^t=1}{\min} J(H_{0i})$. Then for the non-marginal multiple testing procedure the following hold almost surely:
	\begin{align*}
	\lim_{n\rightarrow\infty}\frac{1}{n}\log \mfdr =-J_{\min};\quad
	\lim_{n\rightarrow\infty}\frac{1}{n}\log \fdrx =-H_{\min}.
	\end{align*}
\end{theorem}

Notably, both $J_{\min}$ and $H_{\min}$ are positive and hence the posterior $FDR$ along with its modified version converge to 0 at an exponential rate with increasing sample size. Interestingly, the convergence rate is in terms of the KL-divergence between the true data generating process $P$ and the class of postulated models $F_{\btheta}$. We see that the posterior error measures have this very interesting property where they truly indicate the divergence from the true data-generating process. 

\begin{remark}
	\label{rem:known_group}
	Even though $\nmd$ is asymptotically consistent for data-dependent group formations, asymptotic convergence rate of $\mfdr$ may not hold in such case. For increasing sample size the group structures may change, resulting in ambiguity in the definition of $\mfdr$. Therefore in Theorem \ref{theorem:preferability} we assume that the group structures are known {\it a priori}.
\end{remark}

So far we have investigated the asymptotic properties of $\mfdr$ and $\fdrx$, which is a valid exercise from the Bayesian perspective, as the data is conditioned upon
in these error measures. We now study the asymptotic properties of $mpBFDR$ and $pBFDR$. $mpBFDR$ is defined as
\begin{equation}
mpBFDR=E_{\bX_n}\left[\mfdr|\delta_{\mathcal M}(\bzero|\bX_n)=0\right],
\label{eq:mpbfdr2}
\end{equation}
where $\bzero$ is the decision configuration that no null hypothesis is rejected.
$pBFDR$ is where the expectation in \eqref{eq:mpbfdr2} is of $\fdrx$. Indeed, expectations
of the error measures are traditionally more popular in multiple testing. The following theorem provides the asymptotic results of $mpBFDR$ and $pBFDR$.
\begin{corollary}
	\label{bfdr_lm}
	Under conditions \ref{A1}-\ref{A2}, for the non-marginal procedure we have
	\begin{align*}
	\lim_{n\rightarrow\infty} mpBFDR =0;\quad
	\lim_{n\rightarrow\infty} pBFDR =0.
	\end{align*}
\end{corollary}
It is important to remark that although the aforementioned expected error measures converge to zero as shown by Corollary \ref{bfdr_lm}, it does not seem to be possible to 
obtain the rates of convergence to zero in general, as in Theorem \ref{theorem:preferability} associated with the corresponding posterior versions.

As discussed in Section \ref{subsec:asymp_robust}, the non-marginal method is robust in the sense that it is consistent for any group structure. However, the convergence rate of the $\mfdr$ shows that it takes into account the dependence among hypotheses through the group structures. Hence, it may lose its effectiveness over $\fdrx$ in case the group choice is injudicious. In practical situations, where sample size is fixed, thoughtful choice of groups is very important.


\subsection{Asymptotic properties of versions of $FNR$}
\label{sec:compare_BFNR}


As in the case of $FDR$, similar results can also be derived for different versions of $FNR$. We state the result in the following theorem.
\begin{theorem}
	\label{corollary:limit_bfnrs}
	Assume conditions \ref{A1} and \ref{A2}. Let $\tilde H_{\min}=\underset{i:d_i^t=0}{\min}  J(H_{1i})$. Then for the non-marginal multiple testing procedure
	\begin{equation}
	\lim_{n\rightarrow\infty}\frac{1}{n}\log \fnrx =-\tilde H_{\min}.\label{eq:H_min}
	\end{equation}
\end{theorem}
Thus we see that or the non-marginal method, $\fnrx$ do asymptotically diminish to zero exponentially fast with the convergence rate being directly proportional to the KL-divergence between from true data generating process. 
$pBFNR$ is defined as follows:
\begin{equation*}
pBFNR=E_{\bX_n}\left[\fnrx|\delta_{\mathcal M}(\bone|\bX_n)=0\right],
\end{equation*}
where $\bone$ is the decision configuration that no null hypothesis is accepted. The following asymptotic result holds for $pBFNR$ as well 
\begin{corollary}
	\label{bfnr_lm}
	Under conditions \ref{A1}-\ref{A2}, for the non-marginal procedure we have
	\begin{equation*}	
	\lim_{n\rightarrow\infty} pBFNR =0.
	\end{equation*}
\end{corollary}
Although Corollary \ref{bfnr_lm} asserts convergence of the relevant versions of $BFNR$ to zero, it does not seem to be possible to provide their rates of convergence,
as in $\fnrx$. This issue is in keeping with the corresponding versions of $BFDR$.

\begin{remark}
	It is proper to envisage possible modification of $FNR$ with respect to the new notions of errors. In Section \ref{sec:mpBFNR} we show that, under a mild assumption, the asymptotic convergence rates of $\fnrx$ and its modified counterpart, are equal. Therefore, in this main article, we continue the relevant discussions with respect to the existing versions of $FNR$ only.
\end{remark}

\section[Controlling type-I error]{Convergence of $FNR_{\bX_n}$ and $BFNR$ when versions of $BFDR$ are $\alpha$-controlled}
\label{sec:asymp_category_a}

We now enforce asymptotic control over $mpBFDR$ and $pBFDR$ in the sense that they 
converge to $\alpha\in (0,a)$, instead of zero, for some $0<a\leq 1$, and study the asymptotic behaviour
of $pBFNR$. Here it is important to point out that \ctn{chandra2017} proved 
that for both $\nmd$ and additive loss-function based methods,
$mpBFDR$ and $pBFDR$ are continuous and non-increasing in $\beta$ and therefore $\beta$
can be tuned to set the type-I errors at any desired size $\alpha$. However, as we show in the asymptotic case,
it is not possible to incur too high type-I error, that is, $a$ can not be arbitrarily close to 1. This is not unexpected, since consistent
methods can not commit arbitrarily large errors asymptotically. Naturally the question arises whether $\alpha$-control
of versions of $BFDR$ is at all necessary. The answer is that- since it is a standard practice in multiple testing
to exercise $\alpha$-control on versions of $FDR$ in order to incur lesser type-II error, 
it is important to investigate what would be the feasible range of values of $\alpha$ to attain in large or even moderately large samples, and for such $\alpha$'s how the type-II error would behave. 
We attempt to address these questions with respect to the non-marginal procedure and additive loss-function based method.

\subsection{Convergence of $mpBFDR$ and $pBFDR$ to $\alpha$ for $\nmd$}
\label{subsec:constant_BFDR}

We begin with the following theorem that provides the bound for the maximum $mpBFDR$ that can be incurred asymptotically. 
\begin{theorem}
	\label{theorem:mpBFDR_alpha}
	In addition to \ref{A1}-\ref{A2}, assume the following: 
	\begin{enumerate}[label={(B\arabic*)}]	
		\item \label{ass_groups} 
		Let each group of a particular set of $m_1~(<m)$ groups out of the total $m$ groups be associated with at least one false null hypothesis, and that
		all the null hypotheses associated with the remaining $m-m_1$ groups be true. Let us further assume that the latter $m-m_1$ groups do not have any overlap with the remaining $m_1$ groups. 
		Without loss of generality assume that $G_1,\ldots,G_{m_1}$ are the groups each consisting of at least one false null and $G_{m_1+1},G_{m_1+2},\cdots,G_{m}$ are the groups
		where all the null hypotheses are true. 	
	\end{enumerate}
	Then the maximum $mpBFDR$ that can be incurred, asymptotically lies in $\left( \frac{1}{\sum_{i=1}^md_i^t+1}, \frac{m-m_1}{\sum_{i=1}^md_i^t+m-m_1} \right)$.
\end{theorem}
\begin{remark}
	\label{remark_contradict}
	The proof of Theorem \ref{theorem:mpBFDR_alpha} crucially uses the result 
	that $mpBFDR$ is non-increasing with $\beta$. It can be easily seen that this monotonicity with respect to $\beta$ holds for $mFDR_{\bX_n}$ as well. Hence Theorem \ref{theorem:mpBFDR_alpha}
	is also valid for $mFDR_{\bX_n}$.
\end{remark}

\begin{remark}
	\label{remark:single_G}
	Theorem \ref{theorem:mpBFDR_alpha} holds when $G_i\subset\{1,\ldots,m\}$ for at least one $i\in\{1,\ldots,m\}$. But if
	$G_i=\{1,\ldots,m\}$ for $i=1,\ldots,m$, then $mpBFDR\rightarrow 0$ as $n\rightarrow\infty$, for any
	sequence $\beta_n\in[0,1]$. This is because in this case there does not exist any $\bd\neq\bd^t$ such that
	$$P\left(\sum_{i=1}^md_iw_{in}(\bd)-\sum_{i=1}^md^t_iw_{in}(\bd^t)
	>\beta_n\left(\sum_{i=1}^md_i-\sum_{i=1}^md^t_i\right)\right)>0,$$
	as $n\rightarrow\infty$.
\end{remark}


Theorem \ref{theorem:mpBFDR_alpha} also clarifies that for any arbitrary configuration of groups, it is not possible to commit arbitrarily large error when the sample size is large enough. The joint structure provides a safeguard against incurring large errors. However in practical situations dealing with real life data, it is common practice to control type-I error at some pre-specified level $\alpha~(>0)$ both in single and multiple hypothesis testing problems, which renders the very important task of investigating the feasible range of $\alpha$. 

In this regard, \ref{ass_groups} is the condition under which possible values of type-I error to be controlled are available, at least for large $n$. Note that to incur type-I error it is required to reject some true null hypotheses. As the grouping structures prevent from committing arbitrary error by the non-marginal procedure, \ref{ass_groups} is required. By virtue of this condition, there are some true null hypotheses in isolation which can be rejected. In the following theorem we provide an asymptotic bound on the maximum type-I error that can be incurred.

\begin{theorem}
	\label{corollary:beta_n0}
	
	Assume condition \ref{ass_groups} and let $mpBFDR_\beta$ denote the procured $mpBFDR$ in the non-marginal procedure where the penalizing constant is $\beta$. Suppose
	\begin{align}
	\lim_{n\rightarrow\infty} mpBFDR_{\beta=0}= E.\label{eq:lim_mpbfdr}
	\end{align}
	Then, for any $\alpha<E$ and $\alpha\in\left(\frac{1}{\sum_{i=1}^md^t_i+1},\frac{m-m_1}{\sum_{i=1}^md_i^t+m-m_1} \right)$, 
	there exists a sequence $\beta_n\rightarrow 0$ such that $mpBFDR_{\beta_n}\rightarrow\alpha$ as $n\rightarrow\infty$.
\end{theorem}

Since $mpBFDR$ is decreasing in $\beta$, $\beta$ can be interpreted as a balance provider between type-I and type-II errors. Corollary \ref{bfdr_lm} shows that $mpBFDR$ decays to $0$ when $\liminf_{n\rightarrow\infty}\beta_n>0$ and Theorem \ref{corollary:beta_n0} shows that for $\alpha$-control, we must have $\lim_{n\rightarrow\infty}\beta_n=0$. Since $mpBFDR$ is decreasing in $\beta$, it intuitively indicates that, in the case of $\alpha$-control of 
$mpBFDR$ the sequence $\{\beta_n\}$ has to be dominated by any $\{\beta_n\}$ sequence for which Corollary \ref{bfdr_lm} holds. 
Theorem \ref{corollary:beta_n0} formalizes this intuition and shows that a smaller sequence of $\beta_n$ has to be taken for $\alpha$-control.

From the proofs of Theorem \ref{theorem:mpBFDR_alpha} and \ref{corollary:beta_n0}, it can be seen that replacing $w_{in}(\hat \bd)$ by $v_{in}$ does not affect the results. Hence we state the following corollary.
\begin{corollary}
	\label{theorem:pBFDR_alpha}
	
	Assume condition \ref{ass_groups} and let $pBFDR_\beta$ denote the procured $pBFDR$ in the non-marginal procedure where the penalizing constant is $\beta$. Suppose
	\begin{align*}
	\lim_{n\rightarrow\infty} pBFDR_{\beta=0}= E',
	\end{align*}
	Then, for any $\alpha<E'$ and $\alpha\in\left(\frac{1}{\sum_{i=1}^md^t_i+1},\frac{m-m_1}{\sum_{i=1}^md_i^t+m-m_1} \right)$, 
	there exists a sequence $\beta_n\rightarrow 0$ such that $pBFDR_{\beta_n}\rightarrow\alpha$ as $n\rightarrow\infty$.
\end{corollary}

We now investigate, as special cases of the above results, the situations where $G_i=\{i\}$ for all $i$.
Recall that in this case the additive loss function based methods are special cases of the non-marginal procedure. 
In such cases, $mpBFDR$ also boils down to $pBFDR$.
The following theorem gives the result for asymptotic $\alpha$-control of $pBFDR$ in this situation.
\begin{theorem}
	\label{theorem:pBFDR_alpha2}
	Let $m_0~(<m)$ be the number of true null hypotheses. Then for any $0<\alpha<\frac{m_0}{m}$, there exists a sequence $\beta_n\rightarrow 0$
	as $n\rightarrow\infty$ such that for the additive loss function based methods
	\begin{equation*}
	\underset{n\rightarrow\infty}{\lim}~pBFDR_{\beta_n} =\alpha.
	\end{equation*}
\end{theorem}

In the above, we have noted that $mpBFDR$ reduces to $pBFDR$ when $G_i=\{i\}$ for all $i$. However, for any additive loss function based multiple testing method, we may still envisage
the measure $mpBFDR$ where the definition of $mpBFDR$ considers the adequate dependent structure by means of non-singleton $G_i$'s. This has the advantage of yielding 
non-marginal decisions even though the actual criterion to be optimized is a sum of loss functions
associated with individual parameters and decisions. 
In the following theorem we show that the same asymptotic result as 
Theorem \ref{theorem:pBFDR_alpha2} also holds for $mpBFDR$ in the case of additive loss functions, without assumption \ref{ass_groups}.
\begin{theorem}
	\label{theorem:mpBFDR_alpha2}
	Let $\alpha$ be the desired level of significance where $0<\alpha < \frac{m_0}{m}$, $m_0~(<m)$ being the number of true null hypotheses. Then there exists a sequence $\beta_n\rightarrow 0$
	as $n\rightarrow\infty$ such that for the additive loss function based method
	\begin{equation*}
	\underset{n\rightarrow\infty}{\lim}~mpBFDR_{\beta_n}=\alpha.
	\end{equation*}
\end{theorem}

It is interesting that for the additive loss function based method, Theorem \ref{theorem:mpBFDR_alpha2} holds without condition \ref{ass_groups}. This condition is an added imposition 
to study the theoretical properties of the non-marginal procedure when $mpBFDR$ is controlled at level $\alpha$. \ref{ass_groups} ensures that there are some isolated groups of hypotheses. 
Although there is no notion of grouping in the additive loss function, as we pointed out above, $mpBFDR$ does correspond to groups that are not singletons. 
However, $mpBFDR (\mathcal M)\geq pBFDR (\mathcal M)$ for any multiple testing method $\mathcal{M}$, for arbitrary sample size, and this crucially ensures that
the result asserted by Theorem \ref{theorem:mpBFDR_alpha2} goes through even without \ref{ass_groups}. 

\begin{remark}
	Note that Theorems \ref{corollary:beta_n0}-\ref{theorem:mpBFDR_alpha2} and Corollary \ref{theorem:pBFDR_alpha} 
	 use continuity of the expected versions of $FDR$ with respect to $\beta$, in addition to their non-increasing nature with respect to $\beta$.
	The continuity property need not be satisfied by the corresponding Bayesian versions given the data, and hence we can not assert that the aforementioned results continue to hold
	for the corresponding Bayesian versions of $FDR$ (conditional on the data).
\end{remark}

\begin{remark}
	We have already discussed in the context of Theorems \ref{theorem:mpBFDR_alpha} and \ref{corollary:beta_n0} that condition \ref{ass_groups} is crucial for $\alpha$-control 
	for the non-marginal method, and without the assumption, $mpBFDR$ would diminish to zero asymptotically. This signifies that it is difficult to commit errors by the non-marginal method thanks to its
	dependence structure, so that extra assumption is needed for positive $\alpha$-control. On the other hand, Theorems \ref{theorem:pBFDR_alpha2} and \ref{theorem:mpBFDR_alpha2} show that
	for other multiple testing methods based on additive loss, $\alpha$-control is possible without \ref{ass_groups}, not only with respect to $pBFDR$ but also with respect to $mpBFDR$,
	which includes the dependence structure in its definition. 
It certifies that if the underlying multiple testing procedure does not consider dependence, then however sensible the underlying model is, 
	the errors can be larger compared to the non-marginal procedure.
\end{remark}

\subsection{Asymptotic properties of type-II errors  when $mpBFDR$ and $pBFDR$ are asymptotically controlled at $\alpha$}
\begin{theorem}
	\label{theorem:mpBFNR_alpha}
	Assume condition \ref{ass_groups}. Then for asymptotic $\alpha$-control of $mpBFDR$ in the non-marginal procedure the following holds almost surely:
	\begin{equation*}
	\limsup_{n\rightarrow\infty} \fnrx \leq -\tilde H_{\min}.
	\end{equation*}
\end{theorem}

\begin{corollary}
	Assume condition \ref{ass_groups}. Then for asymptotic $\alpha$-control of $mpBFDR$ in the non-marginal procedure the following holds: 
	\begin{equation*}
	\lim_{n\rightarrow\infty} pBFNR =0.
	\end{equation*}
\end{corollary}

Thus we see that $pBFNR$ also goes to 0 with increasing sample size when type-I error is asymptotically controlled at $\alpha$. In fact, the posterior type-II error, that is, 
$\fnrx$ converges to zero at a rate faster than or equal to that compared to the case when $\alpha$ control is not imposed. In other words, 
allowing asymptotically non-negligible type-I error may result in lower type-II error. 

\section[Illustration in time-series model] {Illustration of consistency of $\nmd$ in time-varying covariate selection in autoregressive process}
\label{sec:ar1}
Let the true model $P$ stand for the following $AR(1)$ model consisting of time-varying covariates: 
\begin{equation}
x_t=\rho_0 x_{t-1}+\sum_{i=0}^m\beta_{i0}z_{it}+\epsilon_t,~t=1,2,\ldots,
\label{eq:true_ar1}
\end{equation}
where $x_0\equiv 0$, $|\rho_0|<1$ and $\epsilon_t\stackrel{iid}{\sim}N(0,\sigma^2_0)$, for $t=1,2,\ldots$.
We further assume that for $i=1,\ldots,m$, the time-varying covariates $\left\{z_{it}:t=1,2,\ldots\right\}$ are realizations 
of some asymptotically stationary stochastic process.
We set $z_{0t}\equiv 1$ for all $t$.

Now let the data be modeled by the same model as $P$ but with $\rho_0$, $\beta_{i0}$ and $\sigma^2_0$ be replaced with the unknown quantities $\rho$, $\beta_i$
and $\sigma^2$, respectively, that is,
\begin{equation}
x_t=\rho x_{t-1}+\sum_{i=0}^m\beta_iz_{it}+\epsilon_t,~t=1,2,\ldots,
\label{eq:modeled_ar1}
\end{equation}
where we set $x_0\equiv 0$, $\epsilon_t\stackrel{iid}{\sim}N(0,\sigma^2)$, for $t=1,2,\ldots$.
As in $P$, we assume that for $i=1,\ldots,m$, the time-varying covariates are realizations of some asymptotically stationary stochastic process. 
For notational convenience, we define $\bz_t=(z_{0t},z_{1t},\ldots,z_{mt})^\prime$, $\bbeta_0=(\beta_{00},\beta_{10},\ldots,\beta_{m0})^\prime$ and $\bbeta=(\beta_0,\beta_1,\ldots,\beta_m)^\prime$. 

For our asymptotic theories regarding the multiple testing methods
that we consider in our main manuscript, we must verify the assumptions of Shalizi for the modeling setups (\ref{eq:true_ar1}) and (\ref{eq:modeled_ar1}), with
$\btheta=(\rho,\beta_0,\beta_1,\ldots,\beta_m,\sigma)$. 
As regards the parameter space, let $\rho\in\mathbb R$, where $\mathbb R$ is the real line,
$\bbeta\in\mathbb R^m$ 
and $\sigma\in\mathbb R^+$, where $\mathbb R^+$ is the positive part of the real line. 
Thus, $\bTheta=\mathbb R^{m+1}\times\mathbb R^+$, is the parameter space. 
We consider any prior on $\bTheta$ that is dominated by the Lebesgue measure, with mild condition on the moments.

With respect to the above setup, we consider the following multiple-testing framework:
\begin{align}
&H_{01}:|\rho|<1\text{ versus }H_{11}:|\rho|\geq 1\text{ and}\nonumber\\
&H_{0i}:\beta_i\in\mathcal N_0\text{ versus }H_{1i}:\beta_i\in\mathcal N^c_0,~\text{ for }~i=2,\ldots,m+1,
\label{eq:ar1_test}
\end{align}
where $\mathcal N_0$ is some neighbourhood of zero and $\mathcal N^c_0$ is the complement of the neighbourhood in the corresponding parameter space.

Verification of consistency of our non-marginal procedure amounts to verification of assumptions \ref{shalizi1}--\ref{shalizi7} for the above setup. 
In this regard, we make the following assumptions regarding the true model and prior distribution: 
\begin{enumerate}[label={(C\arabic*)}]	
	\item \label{ar1} 
	\begin{align}
	&\frac{1}{n}\sum_{t=1}^n\bz_t \rightarrow \bzero;\nonumber\\
	&\frac{1}{n}\sum_{t=1}^n\bz_{t+k}\bz^\prime_t \rightarrow \bzero~\mbox{(null matrix)},~\mbox{for any}~k\geq 1;\label{eq:ass2}\\
	&\frac{1}{n}\sum_{t=1}^n\bz_t\bz'_t\rightarrow \bSigma_z,\nonumber
	\end{align}
	as $n\rightarrow\infty$.
	
	\item \label{ar2} $\underset{t\geq 1}{\sup}~|\bz_t^\prime\bbeta_0|<C$, for some $C>0$.  
	
	\item \label{ar3} $\thetao$ is an interior point of $\bTheta$.
\end{enumerate}

With these model assumptions, we have to verify the seven assumptions in Section \ref{subsec:assumptions_shalizi} in order to show consistency. Theorem \ref{th:shalizi} essentially tells that under certain model and prior assumptions, the posterior distribution asymptotically concentrates around the true data generating process. In this problem, we need to show that the posterior distribution concentrates around $\thetao$.

An important concept related to the posterior convergence theory is the asymptotic equipartition property, which needs to hold for this model. This is ensured by conditions 
\ref{shalizi1}-\ref{s3}. \ref{s4} fortifies that the class of postulated models are not completely orthogonal to the true data generating process. The sequence of sets $\{\mathcal G_n\}_{n=1}^\infty$ in condition \ref{s5} is analogous to the method of sieves \ctp{geman1982} which ensures that the behaviour of the posterior distribution on the full parameter space is dominated by its behaviour on the sieves. \ref{s6}, together with \ref{s5}, make sure that the prior probability mass outside the sieve is exponentially small
with the decay rate large enough so that the posterior probability mass outside it also goes to zero. Using the analogy to the sieve again, the interpretation of the assumption is that the convergence of the log-likelihood ratio is sufficiently fast and eventually the convergence is uniform, almost surely. 

To show that Bayesian multiple testing methods are consistent for model \eqref{eq:true_ar1}, we need to verify the conditions in Section \ref{subsec:assumptions_shalizi}. These are shown in Section \ref{subsec:A2} which leads to the following theorem.
\begin{theorem}
	\label{theorem:ar1}
	Under model assumptions \ref{ar1} -- \ref{ar3}, the non-marginal multiple testing procedure for the hypothesis testing problem in \eqref{eq:ar1_test} is consistent. 
\end{theorem}

Needless to mention, all the results regarding the asymptotic convergence rate of different multiple testing error measures will also continue to hold for this setup.

As an aside, from the above results, we also get a method for variable selection problem from the multiple testing approach. We do not require any restriction on the choice of prior distribution except that it has to be a proper probability distribution. We have proved the results for dependent data making it quite general.

\section{Simulation study}
\label{simulation_asymptotic}
In this section we compare the performance of the non-marginal procedure $(\nmd)$ with the widely used Bayesian multiple testing methods of \ctn{muller04} ($\mathrm{MPR}$) and \ctn{SanatGhosh08} ($\mathrm{SZG}$). 
With increasing sample sizes, we study the convergence rates of these methods.
We elaborate the simulation design in the following section.

\subsection{True data generating mechanism}
\label{true_mech_asy}
In the simulation study we take $\rho_0=-0.5$, $\sigma_0^2=1$ and $m=150$. As regards the $m$-dimensional true regression vector $\bbeta_0$, we take 10 randomly chosen components to be non-zero and the rest to be zero. We generate the covariates as the following
\begin{equation}
\bz_1,\ldots,\bz_t\overset{iid}{\sim} \mn(\bzero,\bPhi),\label{eq:genz}
\end{equation}
where $\mn(\bzero,\bPhi)$ denotes a multivariate distribution with mean vector $\bzero$ and dispersion matrix $\bPhi$. In this study
$\bPhi$ is a known positive definite matrix. With these covariates and true set of parameters $\btheta_0=(\bbeta_0,\sigma_0,\rho_0)$, we generate the observation $x_1,\ldots,x_t$ following the model in \eqref{eq:true_ar1}. 

\subsection{The postulated Bayesian model}
\label{postmod_asy}
Since most of the true $\beta_{0i}$s are zero, we consider the following global local shrinkage prior similar to \ctn{ishwaran2005} over the $\beta_i$s:
\begin{align*}
\beta_i|\gamma_i&\overset{iid}{\sim} \gamma_i N(0,\tau_i^2)+(1-\gamma_i) N(0,v\tau_i^2),\\
\tau_i&\overset{iid}{\sim} IG(a_0,b_0),\\
v&\sim~C^{+}(0,1),\\
\gamma_i|p&\overset{iid}{\sim} Bernoulli(p),\\
p&\sim~Beta(a_1,b_1),
\end{align*}
where $C^{+}(0,1)$ is the Cauchy distribution restricted on the positive real line and $IG(\cdot,\cdot)$ denotes a \textit{Inverse-gamma} distribution. Similar prior has previously been considered in variable selection problem from a multiple testing perspective by \ctn{dghosh2006}. Here $\gamma_i$s are the allocation variables signifying whether the $i$-th variable is included in the model or not. It is a common practice to work with the allocation variables in Bayesian variable selection problems \ctp{narisetty2014} and therefore we reframe the hypothesis testing problem in \eqref{eq:ar1_test} as the following:
$$H_{0i}:\gamma_i=0\text{ versus }H_{1i}:\gamma_i=1,~\text{ for }~i=2,\ldots,m+1.$$

$\tau_i$s are positive numbers taking into account the uncertainty of $\beta_i$s being non-zero when $\gamma_i=1$. $v$ is a very small quantity allocating very high probability around 0 when $\gamma_i=0$. 
We have adjusted $a_1$ and $b_1$ such that the mode of the prior Beta distribution of $p$ is $0.1$. As regards $\sigma$ and $\rho$, we consider the following distributions as prior for these parameters:
\begin{align*}
\sigma^2&\sim IG (a_2,b_2),\\
\rho&\sim N (0,1),
\end{align*}
Here $a_2$ and $b_2$ are adjusted such that the mode of the the prior distribution is 1 and variance 100. For all the three methods, the same prior distribution is considered. 

For implementation of the $\nmd$ method in this simulation study groups are formed according to the strategy in Section \ref{subsec:choiceofgroups} where $\Lambda$ is taken to be the posterior correlation matrix of $\bbeta$ computed the MCMC samples. With these groups, we implement the non-marginal method. 

\subsection{Criteria for comparing different multiple testing methods in this study}
\label{subsec:comp_criteria}
Different multiple testing methods are expected to yield different decision configurations for the same given dataset. We adopt three different criteria for comparing 
the performances of the competing multiple testing procedures, which we briefly discuss below.

Let $\bd_{\mm}$ be the decision configuration obtained by a multiple testing method $\mm$. We compute the Jaccard similarity coefficient 
\ctp{Jaccard01,Jaccard08,Jaccard12} between the true decision configuration $\bd_0$ and $\bd_{\mm}$ for each of three multiple testing methods and compare their performances. 

Let $\bbeta_{\mm}$ and $\hat\rho$ be the mode of the posterior distributions of  $\bbeta$ and $\rho$, respectively, given the data. We also compute the Euclidean distance between $(\bbeta_0,\rho_0)$ and $(\bbeta_{\mm},\hat\rho)$. In this context, note once the multiple testing procedure identifies the significant covariates, we no longer consider
the shrinkage prior for $\beta_i$ for computing the posterior distributions of $\btheta$ and $\rho$, but set $\beta_i\overset{iid}{\sim} N(0,\tau^2)$.

With the significant covariates and a future covariate $\bz_{t+1}$, we compute the posterior predictive distribution of $x_{t+1}$ and compute the Kolmogorov-Smirnov (KS) distance from the true predictive distribution of $x_{t+1}$. Again, we consider $\beta_i\overset{iid}{\sim} N(0,\tau^2)$. 

In other words, we compare the performance and accuracy of the three competing Bayesian multiple testing methods by means of the Jaccard similarity coefficient, Euclidean distance and KS-distance. For five different sample sizes, we replicate our simulation experiments 750 times and compare the boxplots.

Notably, for all the three competing Bayesian multiple testing methods, $\fdrx$ is controlled at level 0.05.

\subsection{Comparison of the results}
\label{subsec:sim_comp}
\begin{figure}
	\centering
	\subfloat[][] {{\includegraphics[scale=.6]{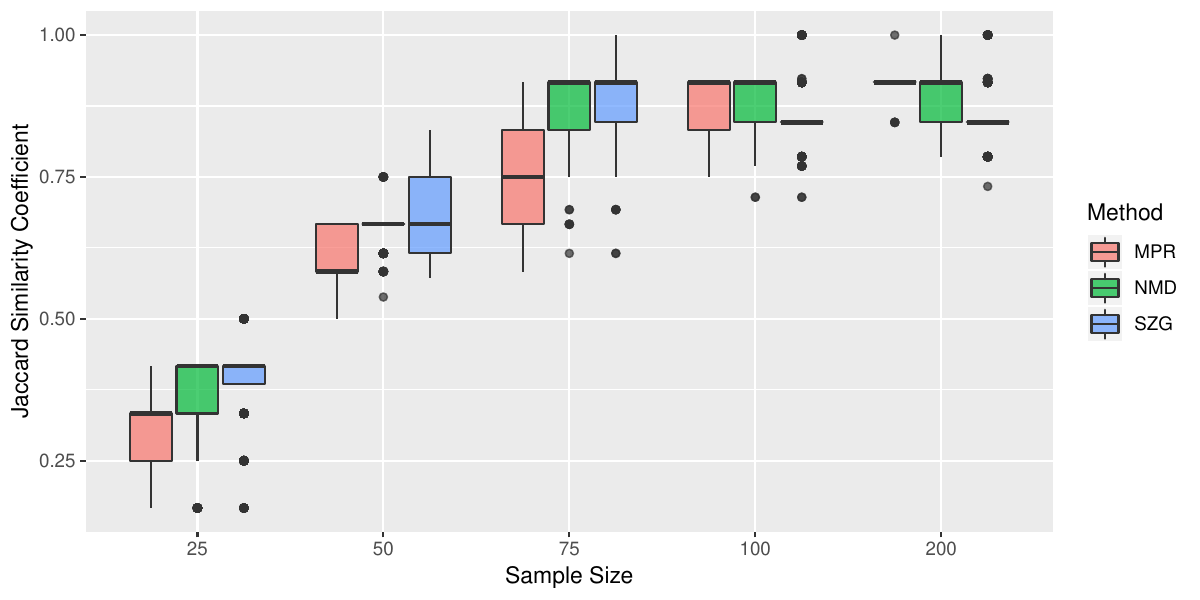} }\label{jacc}}\\
	\subfloat[][] {{\includegraphics[scale=.6]{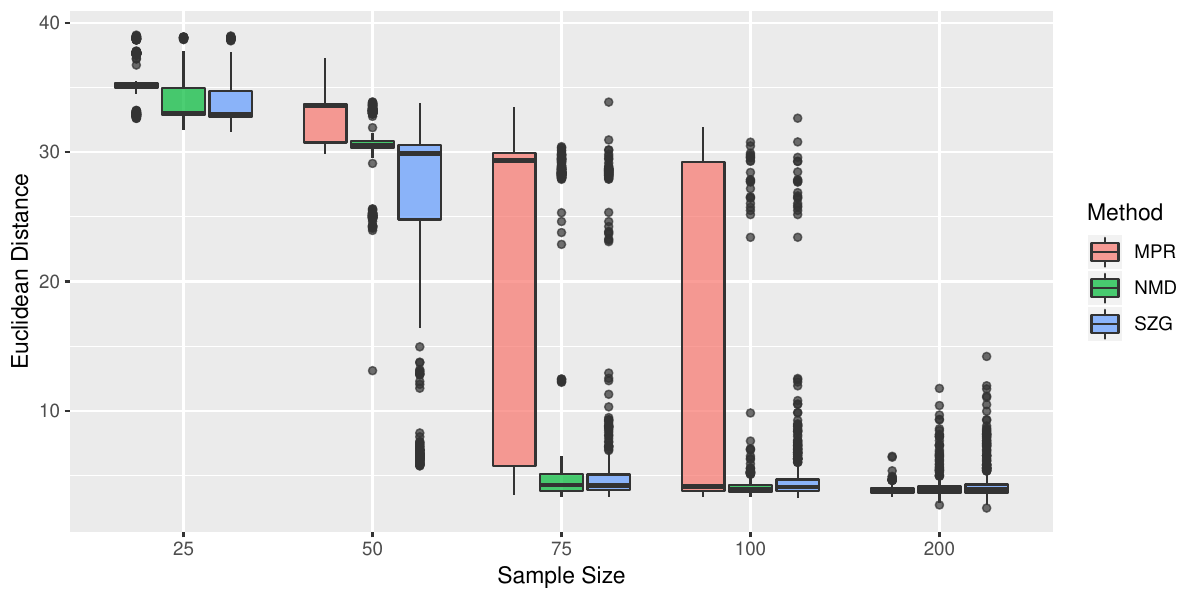} }\label{norm}}\\
	\subfloat[][] {{\includegraphics[scale=.6]{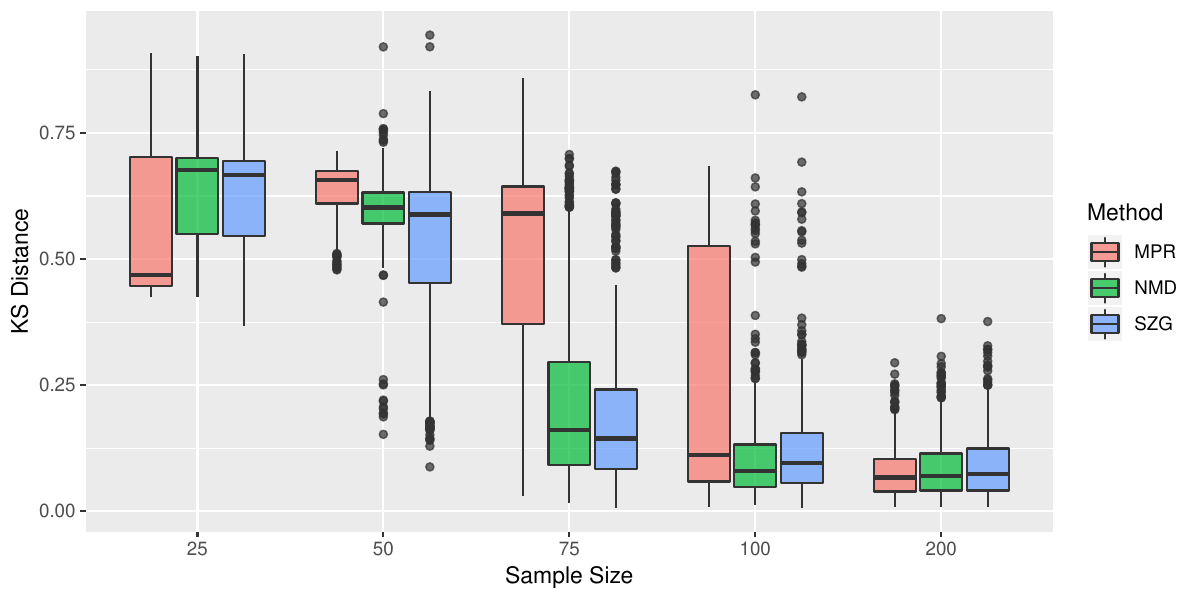} }\label{ks}}
	\caption{Performance comparison through boxplots. }
	\label{fig:compare_Bayesian_asymp}
\end{figure}

From Figure \ref{jacc}, we see that the Jaccard Similarity Coefficients have stabilized near 1 sample size 75 onward indicating that the asymptotic theory is indeed takes precedence for all the methods, when the sample size gets sufficiently large. Interestingly, the $\nmd$ method has the fastest convergence rate with respect to sample size in terms of accurately detecting the truly significant covariates and also exhibits the best performance when the sample sizes are small. Similar behaviour can be observed with respect to the Euclidean distance from the true parameter values (see Figure \ref{norm}). As regards the KS distances depicted in Figure \ref{ks}, we can see that the results of the $\nmd$ are the most stable for every sample size, and with moderately large sample size this method gives the best performance. In this study, greater accuracy of the $\nmd$ method, particularly for small sample size, indicates that in practical multiple hypothesis testing applications where the sample size is generally much smaller as compared to the number of parameters, incorporating the dependence structure in the multiple testing method indeed boosts accuracy.

Observe that variability is much higher in the Euclidean and KS distances compared to Jaccard similarity coefficients. 
Figure \ref{jacc} vindicates that as we are observing more and more samples the right regressors are getting selected with increasing precision. Nonetheless incorrect decision regarding some regressors, even with moderately high regression coefficient, would contribute significantly to the Euclidean and KS distances. This is reflected in Figures \ref{norm} and \ref{ks}.

\subsection{Empirical studies on model misspecfication}
\label{subsec:sim_misspecification}
In this section we study the effect of model misspecification on multiple testing methods. We generate data from the $AR(1)$ model in \eqref{eq:true_ar1} for varying values of $\rho_0$. To allow model misspecification we ignore the autoregressive part while fitting the data and perform variable selection according to the global local shrinkage prior in Section \ref{postmod_asy}. The true values of the parameters are same as we have considered in Section \ref{true_mech_asy} with sample size of $n=100$. The different values of $\rho_0$ are provided in the x-axis of the different panels in Figure \ref{fig:misspecification}. We compute the Jaccard similarity coefficient, Euclidean norm and KS distance in the same way as described in Section \ref{subsec:comp_criteria}.
\begin{figure}
	\centering
	\subfloat[][] {{\includegraphics[scale=.6]{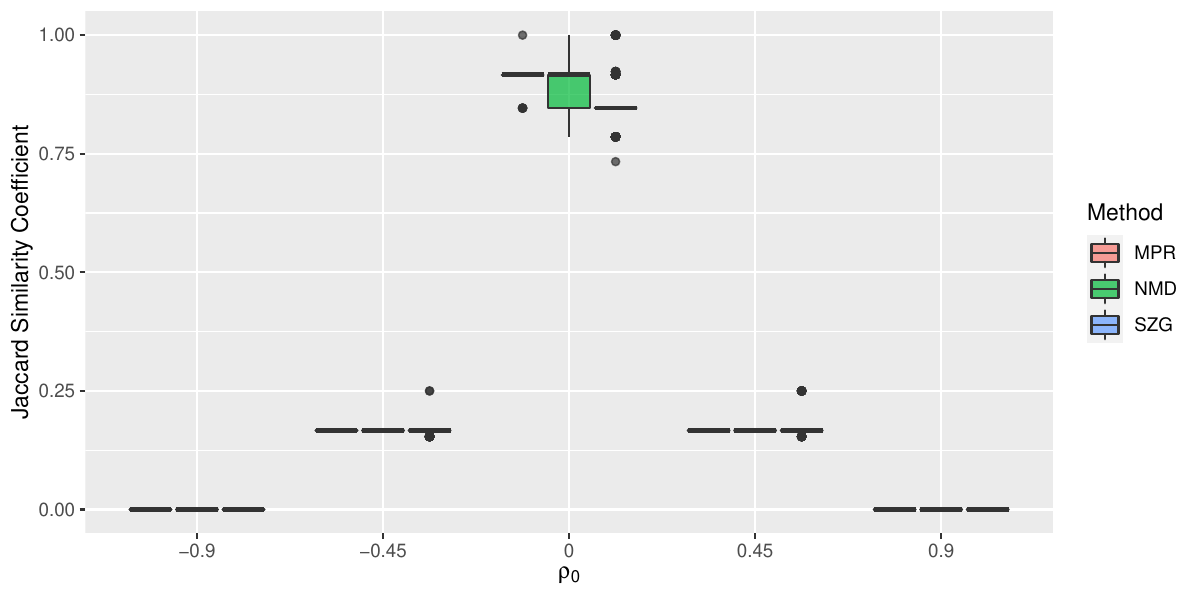} }\label{misspecification_jacc}}\\
	\subfloat[][] {{\includegraphics[scale=.6]{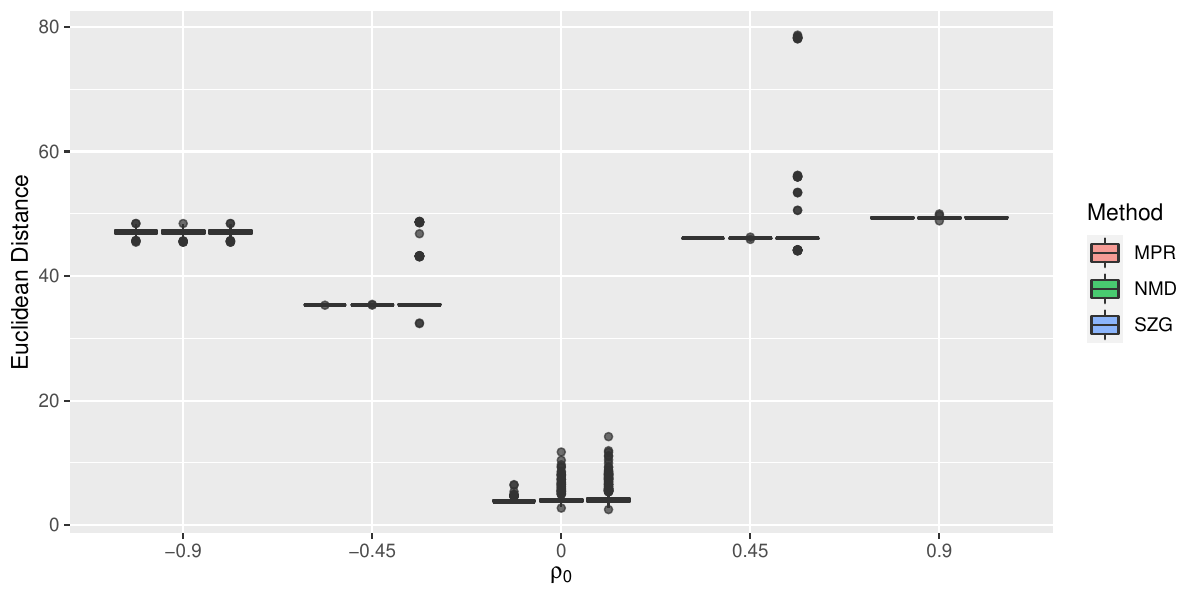} }\label{misspecification_norm}}\\
	\subfloat[][] {{\includegraphics[scale=.6]{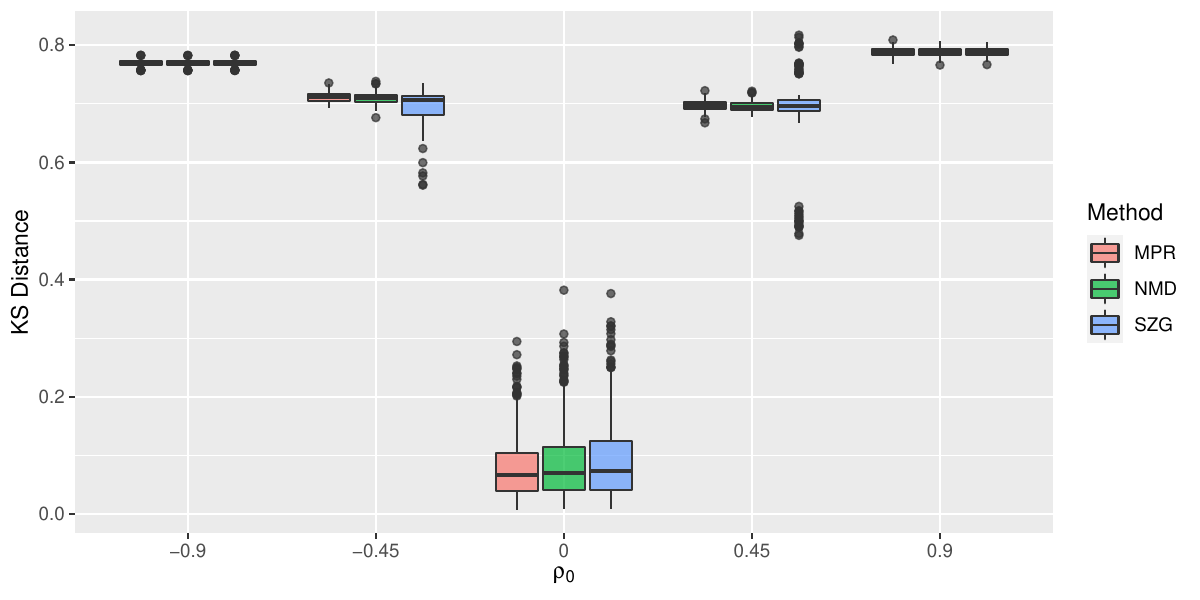} }\label{misspecification_ks}}
	\caption{Effect of model misspecification on multiple testing methods}
	\label{fig:misspecification}
\end{figure}

Note that for $\rho_0=0$ all the methods perform quite accurately. In this case there is no autoregressive component in the true data generating model. Also Figure \ref{fig:compare_Bayesian_asymp} shows that asymptotics is taking precedence from sample size 75 onward. As the performance of all the three competing methods depend upon 
appropriate posterior probabilities, accurate results are quite expected for $\rho_0=0$. Variability in the Euclidean norms and KS distances is much lesser here compared to Figures \ref{norm} and \ref{ks} for $n=100$. The added precision is not surprising as we do not have the autoregressive component to model here.

However, the performance of all the methods deteriorate with increase in model misspecification. The posterior probabilities of events may not properly showcase the uncertainty in case the class of postulated models have a high KL divergence from the true data generating process. As $\rho_0$ deviates from zero model misspecification increases (see Lemma \ref{th:KL_ar} from the supplement section) and the Bayesian multiple testing methods under consideration, being based on posterior probabilities, fail to perform
adequately. The issue is apparent from Figure \ref{fig:misspecification}. This study highlights that for misspecified models which are inadequate for explaining the variability in the data, it is indeed difficult to extract meaningful inference out of them.

\section{Real data analysis}
\label{sec:realdata}
We now consider variable selection using our Bayesian non-marginal multiple testing method in a real data context. 
The data, available at \url{https://www4.stat.ncsu.edu/~boos/var.select/maize.html},
obtained from \ctn{Buckler09}, is regarding 25 crosses (also called families or populations) of maize flowers, each with about 200 observations
on recombinant inbred lines (RILs). There
are 7389 independent variables (covariates) representing the SNP markers, and the response variable is 
``days to anthesis male flowering time" (dtoa). In all there are 4981 observations
for the 25 crosses (excluding the missing values). Our aim to apply the Bayesian non-marginal multiple testing procedure to select the influential marker variables
from the total of 7389, in a linear regression context, for each of the 25 crosses, each having about 200 observed values.

We consider the same Bayesian model as in Section \ref{postmod_asy} for this variable selection problem and subsequently employ our multiple testing procedure to select the relevant SNP markers. With the selected markers, we compute the corresponding fitted values for each of the different populations. Figure \ref{fig:real_data}, displaying the 
observed versus fitted dtoa values for each of the different populations, vindicate that the data variability is adequately explained by our model and methodologies. 
Due to space constraints we show the plots of 12 populations in the main article and the rest in Section \ref{sec:supp_realdata}. In the same latter section, we also report
the causal SNPs for some of the populations.

\begin{figure}[ht]
	\centering
	\includegraphics[width=1\linewidth]{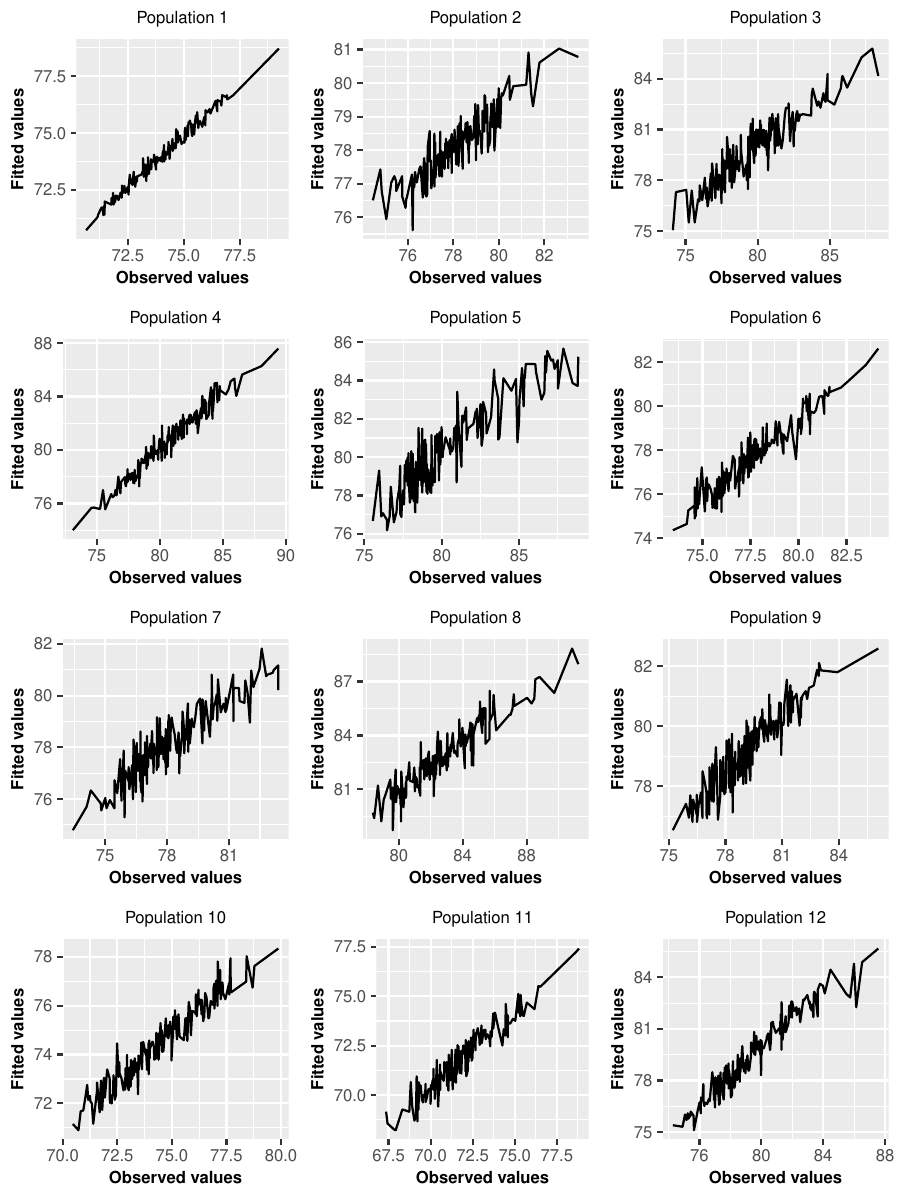}
	\caption{Fitted \textit{dtoa} versus observed \textit{dtoa} }
	\label{fig:real_data}
\end{figure}

\section{Summary and conclusion}
\label{sec:conclusion}
In this article we have investigated asymptotic properties of Bayesian multiple testing procedures. We have shown strong consistency of the non-marginal Bayesian procedure under general dependence structure. As a corollary we have shown that additive loss function based approaches are also consistent.

We have also studied asymptotic properties of multiple testing error rates. We have shown that the posterior versions of the error rates, namely, $\fdrx$ and $\fnrx$, are directly associated with the entropy rate of the true data generating model. Hence, from the Bayesian perspective, we advocate the posterior versions of error rates conditioned on the data. 
In the light of the dependence structure associated with the hypotheses, we introduce $\mfdr$- a modified version of $\fdrx$; the modification being with respect to the dependence among the parameters. The modified version is seen to be associated with a smaller entropy compared to its existing counterpart. 

For $\alpha$-control of type-I errors in the non-marginal procedure, a mild, but still an extra assumption of existence of disjoint groups of hypothesis where the nulls are true, is required. 
However, as we elucidated, this condition indeed indicates that grouping dependent hypotheses pools information across them and provides an extra safeguard against committing error.
Importantly, as we have shown, for large sample sizes, $\alpha$ can not take
any value in $(0,1)$; in particular, we have provided lower bounds to the maximum possible values of $mpBFDR$ and $pBFDR$ and have shown that these lower bounds
are significantly bounded away from $1$, so that setting large values of $\alpha$ is not possible for large samples. Hence for large samples,
the practitioner must choose $\alpha$ carefully.
As regards type-II error, we have shown that, with $\alpha$-control
of type-I error rates, $pBFNR$ is likely to converge to zero at a faster rate than that without $\alpha$-control of the type-I errors. 
Thus the usual expectation of statisticians, that controlling type-I error yields smaller type-II error in single hypothesis testing, 
is expected to hold in our multiple testing framework.

We draw attention to the fact that most of our asymptotic results crucially hinge on the assumptions considered in Section \ref{subsec:assumptions_shalizi}. In this regard,
we have illustrated these assumptions in a variable selection problem with autoregressive response variables from a multiple testing perspective, along with the test for stationarity. In this problem we show that the assumptions hold for any choice of proper prior over the general, non-compact parameter space, entailing strong consistency of Bayesian multiple testing methods. We have also discussed how verification of these assumptions are implicitly related to showing consistency of the maximum likelihood estimator. Indeed, proving strong consistency of Bayesian posterior distributions or maximum likelihood estimators is certainly quite challenging for non-compact parameter spaces and dependent setups, and our approach is probably of independent interest in this respect.

We have backed up our theoretical investigations with extensive simulation studies, comparing the performance of our $\nmd$ method with two other Bayesian multiple
testing procedures for sample sizes ranging from small to moderately large. The results indicate clear superiority of the $\nmd$ method, particularly for small sample sizes.
This is quite encouraging, since in practice, sample sizes are expected to be small compared to the number of available covariates. The message underlying the superior
performance of $\nmd$ is that it exploits the dependence structure in a more wholesome way compared to the existing methods. 

The empirical studies on misspecified models are particularly important. These studies show that multiple testing methods relying on inadequate models would suffer. The results by \citet{Shalizi09} show that asymptotically the model with the minimum KL divergence from the true data generating process would be preferred, however, that preferred model can be quite bad. As the $\nmd$ method reckons on the uncertainty delivered by appropriate posterior probabilities, it suffers in such cases.

Application of our multiple testing procedure to a real maize data concerning selection of influential marker variables from a total of 7389 variables, 
yielded quite encouraging results. Since variable selection from among many variables is an important real problem, our results seem to indicate the importance
of our multiple testing procedure.

In this article we have assumed $m$, the number of hypotheses, to be fixed. But it is also important to investigate the asymptotic theory as $m$ also tends to infinity with 
the sample size $n$, particularly because of its relevance in practical problems. Indeed, we have already made progress regarding this; see \citet{Chandra20}. 
Note that the framework by \citet{Shalizi09} is valid for infinite dimensional models and it has not been too difficult to extend our results in the high-dimensional 
setup with additional mild assumptions. Our high-dimensional asymptotic results on error rates are somewhat less precise than in this current fixed
dimensional setup, in that closed form rates of convergence are not exactly available. We have also extended our current asymptotic results on the $AR(1)$ regression
to high-dimensional asymptotic frameworks.

\section*{Acknowledgement}
We sincerely express our gratitude to the Editor, the Associate Editor and
the referees for their responsible handling of our paper and providing valuable comments that led to significant improvement of the presentation and readability of our paper.


\begin{center}
	{\huge\bf Supplementary Material}
\end{center}

\renewcommand\thefigure{S-\arabic{figure}}
\renewcommand\thetable{S-\arabic{table}}
\renewcommand\thesection{S-\arabic{section}}
\renewcommand{\theequation}{S-\arabic{equation}}

\setcounter{section}{0}
\setcounter{theorem}{0}
\setcounter{figure}{0}
\setcounter{table}{0}

\section[Model assumptions]{Assumptions of \ctn{Shalizi09}}
\label{subsec:assumptions_shalizi}
\begin{enumerate}[label={(S\arabic*)}]	
	\item  \label{shalizi1} Consider the following likelihood ratio:
	\begin{equation}
	R_n(\btheta)=\frac{f_{\btheta}(\bX_n)}{p(\bX_n)}.
	\label{eq:R_n}
	\end{equation}
	Assume that $R_n(\btheta)$ is $\sigma(\bX_n)\times \mathcal T$-measurable for all $n>0$.
	
	\item \label{s2} For each $\btheta\in\Theta$, the generalized or relative asymptotic equipartition property holds, and so,
	almost surely,
	\begin{equation*}
	\underset{n\rightarrow\infty}{\lim}~\frac{1}{n}\log R_n(\btheta)=-h(\btheta),
	\end{equation*}
	where $h(\btheta)$ is given in (S3) below.
	
	\item \label{s3} For every $\btheta\in\Theta$, the KL-divergence rate
	\begin{equation}
	h(\btheta)=\underset{n\rightarrow\infty}{\lim}~\frac{1}{n}E\left(\log\frac{p(\bX_n)}{f_{\btheta}(\bX_n)}\right).
	\label{eq:S3}
	\end{equation}
	exists (possibly being infinite) and is $\mathcal T$-measurable.
	
	\item \label{s4}
	Let $I=\left\{\btheta:h(\btheta)=\infty\right\}$. 
	The prior $\pi$ satisfies $\pi(I)<1$. 
	
	\item \label{s5} There exists a sequence of sets $\mathcal G_n\rightarrow\Theta$ as $n\rightarrow\infty$ 
	such that: 
	\begin{enumerate}
		\item[(1)]
		\begin{equation}
		\pi\left(\mathcal G_n\right)\geq 1-\alpha\exp\left(-\varsigma n\right),~\mbox{for some}~\alpha>0,~\varsigma>2h(\Theta);
		\label{eq:S5_1}
		\end{equation}
		\item[(2)]The convergence in (S3) is uniform in $\theta$ over $\mathcal G_n\setminus I$.
		\item[(3)] $h\left(\mathcal G_n\right)\rightarrow h\left(\Theta\right)$, as $n\rightarrow\infty$.
	\end{enumerate}
	
	For each measurable $A\subseteq\Theta$, for every $\delta>0$, there exists a random natural number $\tau(A,\delta)$
	such that
	\begin{equation}
	n^{-1}\log\int_{A}R_n(\btheta)\pi(\btheta)d\btheta
	\leq \delta+\underset{n\rightarrow\infty}{\limsup}~n^{-1}
	\log\int_{A}R_n(\btheta)\pi(\btheta)d\btheta,
	\label{eq:limsup_2}
	\end{equation}
	for all $n>\tau(A,\delta)$, provided 
	$\underset{n\rightarrow\infty}{\lim\sup}~n^{-1}\log\pi\left(\mathbb I_A R_n\right)<\infty$.
	Regarding this, the following assumption has been made by Shalizi:
	
	\item\label{s6} The sets $\mathcal G_n$ of (S5) can be chosen such that for every $\delta>0$, the inequality
	$n>\tau(\mathcal G_n,\delta)$ holds almost surely for all sufficiently large $n$.
	
	\item \label{shalizi7} The sets $\mathcal G_n$ of (S5) and (S6) can be chosen such that for any set $A$ with $\pi(A)>0$, 
	\begin{equation}
	h\left(\mathcal G_n\cap A\right)\rightarrow h\left(A\right)\text{ as } n\rightarrow\infty 
	\label{eq:S7}
	\end{equation}	
\end{enumerate}

\section{Comparisons of versions of $FNR$}
\label{sec:mpBFNR}
With respect to the new notions of errors in (\ref{eq:tp}) and (\ref{eq:e}), $\fnrx$ can be modified as
\begin{align*}
modified~\fnrx &= \pexp\left[\sum_{\bd\in\mathbb D} \frac{\sum_{i=1}^m(1-d_i)r_iz_i}{\sum_{i=1}^m(1-d_i)\vee1}
\delta_{\mathcal M}\left(\bd|\bX_n\right)
\right]\notag\\
&= \sum_{\bd\in\mathbb D} 
\frac{\sum_{i=1}^{m}(1-d_i)w_{in}(\bd)}{\sum_{i=1}^{m}(1-d_i)\vee1}
\delta_{\mathcal M}(\bd|\bX_n).
\end{align*}
We denote $modified ~\fnrx$ as $\mfnr$. Now, from Theorem \ref{th:shalizi}, $d^t_i=0$ implies
\begin{equation*}
\exp\left[-n\left(J\left(\Theta_{i\bd^t}\right)+\epsilon\right)\right]<w_{in}(\bd^t)
<\exp\left[-n\left(J\left(\Theta_{i\bd^t}\right)-\epsilon\right)\right].
\end{equation*}

Similar to Theorem \ref{corollary:limit_bfnrs}, using the above bounds, we can obtain the asymptotic convergence rate of $\mfnr$, formalized in the following theorem:
\begin{theorem}
	Assume conditions \ref{A1} and \ref{A2}. Let $\tilde J_{\min}=\underset{i:d_i^t=0}{\min} J(\bTheta_{i\bd^t})$. Then for the non-marginal multiple testing procedure
	\begin{equation}
	\lim_{n\rightarrow\infty}\frac{1}{n}\log \mfnr =-\tilde J_{\min}.\label{eq:J_min}
	\end{equation}
\end{theorem}
\begin{proof}
	Following the proof of Lemma \ref{theorem:mpBFDR_form1}, we have 
	\begin{equation*}
	\exp(-n\epsilon)\times\frac{\sum_{i=1}^m(1-d^t_i)e^{-nJ\left(\Theta_{i\bd^t}\right)}}{\sum_{i=1}^m(1-d^t_i)}
	\leq \mfnr
	\leq\exp(n\epsilon)\times \frac{\sum_{i=1}^m(1-d^t_i)e^{-nJ\left(\Theta_{i\bd^t}\right)}}
	{\sum_{i=1}^m(1-d^t_i)},
	\end{equation*}
	from which the proof follows.
\end{proof}

If $J\left(\Theta_{i\bd^t}\right)=J\left(H_{1i}\right)$ for $i=1,\ldots,m$,
it would follow that $w_{in}(\bd^t)$ and $v_{in}$ have the same lower and upper bounds.
Lemma \ref{lemma:J_equal} shows that indeed $J\left(\Theta_{i\bd^t}\right)=J\left(H_{1i}\right)$ for $i=1,\ldots,m$,
under a very mild assumption given by the following.
\begin{enumerate}[label={(A\arabic*)}]	
	\setcounter{enumi}{2}
	\item \label{ass_KL} For any decision configuration $\bd$, define $S(\bd)=\{i:d_i=d_i^t\}$. Then for two decision configurations $\bd$ and $\tilde \bd$, if $S(\bd)\subset S(\tilde{\bd})$, then $J(\bTheta(\bd))>J(\bTheta(\tilde\bd))$. 
\end{enumerate}
Notably in \ref{ass_KL}, $S(\bd)$ is the set of correct decisions. Note that $S(\bd)\subset S(\tilde{\bd})$ implies that number of correct decisions is more in $\tilde{\bd}$ compared to $\bd$. Hence, the model directed by $\bd$ should procure greater divergence. This assumption is easily seen to hold in independent cases, and also in dependent models such as multivariate normal. 

\begin{lemma}
	\label{lemma:J_equal}
	Under \ref{ass_KL}, $J\left(\Theta_{i\bd^t}\right)=J\left(H_{1i}\right)$, for all $i$ such that $d_i^t=0$.
\end{lemma}
\begin{proof}
	For  all $i$ such that $d_i^t=0$, define $\bd^{(i)}$, where $d^{(i)}_j=d^t_j$ for all $j\neq i$, and $d^{(i)}_j=1$ and $S_i=\{\bd:d_i=1\}$. Then
	\begin{equation}
	\postp\left(H_{1i}\right)
	=\sum_{\bd\in S_i} \postp\left(H_{1i}\cap\left\{\cap_{j\neq i}H_{d_j,j}\right\}\right),
	\label{eq:J_1}
	\end{equation}
	so that dividing both sides of (\ref{eq:J_1}) by $\postp\left(\Theta_{i\bd^t}\right)$ yields
	\begin{equation}
	\frac{\postp\left(H_{1i}\right)} {\postp\left(\Theta_{i\bd^t}\right)}
	=1+\sum_{\bd\in S_i\backslash \{\bd^{(i)}\}} \frac{\postp\left(H_{1i}\cap\left\{\cap_{j\neq i} H_{d_j,j} \right\} \right)} {\postp\left(\Theta_{i\bd^t}\right)} 
	\label{eq:J_2}
	\end{equation}
	Theorem \ref{th:shalizi} and \ref{ass_KL} together ensures that as $n\rightarrow\infty$, 
	$\frac{\postp\left(H_{1i}\cap\left\{\cap_{j\neq i} H_{d_j,j} \right\} \right)} {\postp \left(\Theta_{i\bd^t} \right)}
	\rightarrow 0$ exponentially fast, for all $\bd\in S_i\backslash\{\bd^{(i)}\}$. Applying this to the right hand side of (\ref{eq:J_2}) yields
	\begin{equation}
	\frac{\postp\left(H_{1i}\right)} {\postp\left(\Theta_{i\bd^t} \right)}\rightarrow 1,
	\label{eq:J_3}
	\end{equation}
	exponentially fast. Now, applying Shalizi's result to $\postp\left(H_{1i}\right)$ and $\postp\left(\Theta_{i\bd^t} \right)$
	it follows that if $J\left(\Theta_{i\bd^t}\right)\neq J\left(H_{1i}\right)$, then (\ref{eq:J_3}) is contradicted.
	Hence, $J\left(\Theta_{i\bd^t}\right)=J\left(H_{1i}\right)$, for $i=1,\ldots,m$.
\end{proof}

From Lemma \ref{lemma:J_equal}, we see that $\tilde J_{\min}=\tilde H_{\min}$. Thus, we get the following result:
\begin{theorem}
	\label{theorem:asymp_BFNR}
	Assume \ref{A1}--\ref{ass_KL}. Then, for the non-marginal multiple testing procedure,
	\begin{equation}
	\lim_{n\rightarrow\infty}\frac{1}{n}\log\left(\frac{\mfnr}{\fnrx}\right) = 0,
	\label{eq:ass4_BFNR}
	\end{equation}
	and
	\begin{equation}
	\lim_{n\rightarrow\infty}\frac{\log\left(\mfnr\right)}{\log\left(\fnrx\right)} = 1.
	\label{eq:ass4_BFNR2}
	\end{equation}
\end{theorem}
\begin{proof}
	Note that,
	\begin{equation*}
	\frac{1}{n}\log\left(\frac{\mfnr}{\fnrx}\right) = \frac{1}{n}\log\left(\mfnr\right)-\frac{1}{n}\log\left(\fnrx\right).
	\end{equation*}
	Now $\frac{1}{n}\log\left(\mfnr\right)\rightarrow-\tilde{J}_{\min}$ and $\frac{1}{n}\log\left(\fnrx\right)\rightarrow-\tilde{H}_{\min}$ as $n\rightarrow\infty$. Again by Lemma \ref{lemma:J_equal}, $\tilde{J}_{\min}=\tilde{H}_{\min}$. This proves (\ref{eq:ass4_BFNR}). The proof of (\ref{eq:ass4_BFNR2}) follows from (\ref{eq:J_min}) and (\ref{eq:H_min}), using 
	$\tilde{J}_{\min}=\tilde{H}_{\min}$. 
\end{proof}

Theorem \ref{theorem:asymp_BFNR} remains true for any 
$\bG=\{G_1,\ldots,G_m\}$. In other words, given that \ref{ass_KL} holds, (\ref{eq:ass4_BFNR}) shows that none of $\mfnr$ or $\fnrx$ is asymptotically preferable over the other, while
(\ref{eq:ass4_BFNR2}) shows that $\log\left(\mfnr\right)$ and $\log\left(\fnrx\right)$ are asymptotically equivalent, irrespective of how the $G_i$'s are formed. 


\section{Proofs of results in Section \ref{sec:compare_BFDR}}
\begin{proof}[Proof of Theorem \ref{th:nmd_consistent}]
Let $\bTheta^{tc}$ be the complement set of $\bTheta(\bd^t)$. Then by virtue of Theorem \ref{th:shalizi} we have
\begin{equation*}
\lim_{n\rightarrow\infty} \frac{1}{n}\log \postp \left( \bTheta^{tc}\right) = -J\left( \bTheta^{tc}\right).
\end{equation*}
This implies that for any $\epsilon>0$, there exists a $n_0(\epsilon)$ such that for all $n>n_0(\epsilon)$
\begin{align*}
&\exp\left[-n\left(J\left(\bTheta^{tc}\right)+\epsilon\right)\right]
< \postp \left( \bTheta^{tc}\right) < \exp \left [-n\left(J\left(\bTheta^{tc}\right)-\epsilon\right) \right]\\	
\Rightarrow& 1-\exp \left [-n\left(J\left(\bTheta^{tc}\right)-\epsilon\right) \right] < \postp \left( \bTheta^t\right) <1- \exp\left[-n\left(J\left(\bTheta^{tc}\right)+\epsilon\right)\right].
\end{align*}
For notational convenience, we shall henceforth denote $J\left(\bTheta^{tc}\right)$ by $J$.

Observe that if $\bd\in\mathbb D^c_{i}$, at least one decision is wrong corresponding to some hypothesis in $G_i$. As $\postp \left( \bTheta^{tc} \right)$ is the posterior 
probability of at least one wrong decision in the 
parameter space, we have
\begin{equation}
w_{in}(\bd)<\postp \left( \bTheta^{tc}\right) < \exp \left [-n\left(J-\epsilon\right) \right].
\label{eq:ineq_start}
\end{equation}
Similarly for $\bd\in\mathbb D_{i}$ and for false $H_{0i}$
\begin{equation}
w_{in}(\bd)>\postp \left( \bTheta^t\right) >1- \exp \left [-n\left(J-\epsilon\right) \right].
\label{eq:ineq2}
\end{equation}

From conditions 
(\ref{eq:liminf_beta}) and (\ref{eq:limsup_beta}), it follows that there exists $n_1$ such that for all $n>n_1$
\begin{align*}
\beta_n&>\underline{\beta}-\delta,\\
\beta_n&<1-\delta,\text{ such that}
\end{align*}
$\underline{\beta}-\delta>0$ and $1-\bar{\beta}>\delta$, for some $\delta>0$.
It follows using this, (\ref{eq:ineq_start}) and (\ref{eq:ineq2}), that
\begin{align*}
\sum_{i:\bd\in \mathbb{D}_i^c }^m d^t_iw_{in}(\bd^t)- \sum_{i:\bd\in \mathbb{D}_i^c}^md_iw_{in}(\bd) &> \left(1- e^{-n(J-\epsilon )}\right) \sum_{i:\bd\in \mathbb{D}_i^c } d^t_i-e^{-n(J-\epsilon )}\sum_{i:\bd\in \mathbb{D}_i^c } d_i,~\mbox{and}\\
\beta_n \left( \sum_{i:\bd\in \mathbb{D}_i^c}^m d^t_i-\sum_{i:\bd\in \mathbb{D}_i^c }^m d_i\right)&< (1-\delta )\sum_{i:\bd\in \mathbb{D}_i^c}^m d^t_i - (\underline\beta-\delta) \sum_{i:\bd\in \mathbb{D}_i^c}^m d_i.
\end{align*}
Now $n_1$ can be appropriately chosen such that $e^{-n(J-\epsilon )}<\min \{ \delta, \underline\beta-\delta\}$. Note that neither $n_0$ nor $n_1$ depends on $m$. 
Hence, for any value of $m$ and for all $n>\max \{ n_0, n_1\}$,
\begin{align*}
&\sum_{i:\bd\in \mathbb{D}_i^c }^m d^t_iw_{in}(\bd^t)- \sum_{i:\bd\in \mathbb{D}_i^c}^md_iw_{in}(\bd)
> \beta_n \left( \sum_{i:\bd\in \mathbb{D}_i^c}^m d^t_i-\sum_{i:\bd\in \mathbb{D}_i^c }^m d_i\right) ,~\mbox{for all}~\bd\neq\bd^t,\text{ almost surely}\\
\Rightarrow &  \lim_{n\rightarrow\infty}\delta_{\mathcal {NM}}(\bd^t|\bX_n)=1,~\mbox{almost surely}. 
\end{align*}

\end{proof}

\section{Additional results to Section \ref{subsec:asymp_mpBFDR} and proofs}
\begin{lemma}
	\label{theorem:mpBFDR_form1}
	Assume conditions \ref{A1} and \ref{A2}. Then for the non-marginal multiple testing procedure and any $\epsilon>0$, there exists $n_0(\epsilon)\geq 1$ such that for 
	$n\geq n_0(\epsilon)$, the following holds almost surely: 
	\begin{align*}
	\exp\left(-n\epsilon\right)\times\frac{\sum_{i=1}^md_i^t e^ {-nJ(\bTheta_{i\bd^t}^c) } }{\sum_{i=1}^md_i^t}
	\leq \mfdr
	\leq \exp\left(n\epsilon\right)\times\frac{\sum_{i=1}^md_i^t e^ {-nJ(\bTheta_{i\bd^t}^c) } }{\sum_{i=1}^md_i^t}.
	\end{align*}
\end{lemma}
\begin{proof}
	Observe that,
	\begin{align}	
	&\mfdr\nonumber\\
	=&\sum_{\bd\neq\bzero}\frac{\sum_{i=1}^md_i (1- w_{in}(\bd))}{\sum_{i=1}^md_i\vee1}
	\dnm\left(\bd|\bX_n\right)\nonumber\\
	=&\frac{\sum_{i=1}^md_i^t (1- w_{in}(\bd^t))}{\sum_{i=1}^md_i^t}
	\dnm \left(\bd^t|\bX_n\right)+\sum_{\bd\neq\bd^t}\frac{\sum_{i=1}^md_i (1- w_{in}(\bd))}{\sum_{i=1}^md_i\vee1}
	\dnm\left(\bd|\bX_n\right). \label{eq:mpbfdrnum}
	\end{align}
	From the proof of Theorem \ref{th:nmd_consistent}, we see that under \ref{A1}, 
	$\dnm(\bd|\bX_n)=0$ for all $\bd\neq\bd^t$. Also under \ref{A2}, $\bd^t\neq\bzero$.
	For any $\epsilon>0$ and $n\geq n_0(\epsilon)$, it follows from (\ref{eq:shalizi1}) and (\ref{eq:shalizi2}) that a lower bound for (\ref{eq:mpbfdrnum})
	is
	\begin{align*}
	L_n&= \frac{\sum_{i=1}^md_i^t e^ {-n(J(\bTheta_{i\bd^t}^c)+\epsilon) } }{\sum_{i=1}^md_i^t} \dnm(\bd^t|\bX_n) =\exp\left(-n\epsilon\right)\times\frac{\sum_{i=1}^md_i^t e^ {-nJ(\bTheta_{i\bd^t}^c) } }{\sum_{i=1}^md_i^t}.
	\end{align*}
	Similarly, an upper bound is given by
	\begin{align*}
	U_n&= \frac{\sum_{i=1}^md_i^t e^ {-n(J(\bTheta_{i\bd^t}^c)-\epsilon ) } }{\sum_{i=1}^md_i^t} \dnm\left(\bd^t|\bX_n\right)= \exp\left(n\epsilon\right)\times\frac{\sum_{i=1}^md_i^t e^ {-nJ(\bTheta_{i\bd^t}^c) } }{\sum_{i=1}^md_i^t}.
	\end{align*}
\end{proof}

Similar asymptotic bounds can also be obtained for $\fdrx$ under the same conditions. We state it formally in the following corollary.
\begin{corollary}
	\label{theorem:pBFDR_form}
	Assume conditions \ref{A1} and \ref{A2}. Then for the non-marginal multiple testing procedure and any $\epsilon>0$ and large enough $n$
	the following holds almost surely: 
	\begin{align*}
	\exp(-n\epsilon)\times \frac{\sum_{i=1}^md^t_i e^{-nJ(H_{0i})}} {\sum_{i=1}^md^t_i}
	\leq \fdrx
	\leq \exp(n\epsilon)\times \frac{\sum_{i=1}^md^t_ie^{-nJ(H_{0i})}} {\sum_{i=1}^md^t_i}.
	\end{align*}
\end{corollary}

\begin{lemma}
	\label{lemma:fnr_bound}
	Assume conditions \ref{A1} and \ref{A2}. Then for the non-marginal multiple testing procedure and any $\epsilon>0$, there exists a natural number $n_1(\epsilon)$ such that for all $n>n_1(\epsilon)$ the following hold almost surely:
	\begin{equation*}
	\exp(-n\epsilon)\times\frac{\sum_{i=1}^m(1-d^t_i)e^{-nJ\left(\Theta_{i\bd^t}\right)}}{\sum_{i=1}^m(1-d^t_i)}
	\leq \mfnr
	\leq\exp(n\epsilon)\times \frac{\sum_{i=1}^m(1-d^t_i)e^{-nJ\left(\Theta_{i\bd^t}\right)}}
	{\sum_{i=1}^m(1-d^t_i)};
	\end{equation*}
\end{lemma}
\begin{proof}
	Note that by Theorem \ref{th:shalizi}, $d^t_i=0$ implies
	\begin{equation*}
	\exp\left[-n\left(J\left(H_{1i}\right)+\epsilon\right)\right]<v_{in}
	<\exp\left[-n\left(J\left(H_{1i}\right)-\epsilon\right)\right].
	\end{equation*}
	From the above bound, similar to the proof of Lemma \ref{theorem:mpBFDR_form1}, we obtain asymptotic bounds of $\fnrx$.
\end{proof}

Note that, \ref{A2} is required for both Lemma \ref{theorem:mpBFDR_form1} and \ref{lemma:fnr_bound} to hold. Without the condition the denominators of the bounds would become zero. 
For proper bounds of the errors and hence for the limits, \ref{A2} is necessary.\\[2mm]
\begin{proof}[Proof of Theorem \ref{theorem:preferability}]
	From Lemma \ref{theorem:mpBFDR_form1} 
	we obtain the following for $n\geq n_0(\epsilon)$,  
	\begin{align*}
	&\exp\left(-n\epsilon\right)\times \frac{\sum_{i=1}^md^t_i e^ {-nJ(\bTheta_{i\bd^t}^c) } } {\sum_{i=1}^md^t_i }
	\leq \mfdr 
	\leq \exp\left(n\epsilon\right) \times \frac{\sum_{i=1}^md^t_i e^ {-nJ(\bTheta_{i\bd^t}^c) } } {\sum_{i=1}^md^t_i }\\
	&\Longleftrightarrow -\epsilon+\frac{1}{n} \log \left( \sum_{i=1}^md^t_i e^ {-nJ(\bTheta_{i\bd^t}^c) } \right) -\frac{1}{n} \log \left( \sum_{i=1}^md^t_i \right) \leq \frac{1}{n} \log \mfdr \\
	&\qquad\qquad\qquad\qquad\qquad\qquad\qquad\qquad \leq \epsilon + \frac{1}{n} \log \left( \sum_{i=1}^md^t_i e^ {-nJ(\bTheta_{i\bd^t}^c) } \right)-\frac{1}{n} \log \left( \sum_{i=1}^md^t_i \right).
	\end{align*}
	Applying \textit{L'H\^{o}pital's rule} we observe that
	\begin{equation*}
	\underset{n\rightarrow\infty} {\lim} \frac{1}{n} \log \left( \sum_{i=1}^md^t_i e^ {-nJ(\bTheta_{i\bd^t}^c) } \right)= -J_{\min}.
	\end{equation*}
	As $\epsilon$ is an arbitrarily small positive quantity, we have
	\begin{equation*}
	\underset{n\rightarrow\infty} {\lim} \frac{1}{n} \log \mfdr  =- J_{\min}.
	\end{equation*}
	Proceeding in the exact same way, using Corollary \ref{theorem:pBFDR_form}, we obtain
	\begin{equation*}
	\underset{n\rightarrow\infty} {\lim} \frac{1}{n} \log \fdrx = -H_{\min}.
	\end{equation*}
\end{proof}\\[2mm]

\begin{proof}[Proof of Corollary \ref{bfdr_lm}]
	Note that
	\begin{align*}
	mpBFDR =& E_{\bX_n} \left[\sum_{\bd\in\mathbb{D}} \frac{\sum_{i=1}^{m}d_i(1- w_i(\bd) )}{\sum_{i=1}^{m}d_i}
	\delta_\beta(\bd|\bX_n)\bigg{|}\dnm(\bd=\bzero|\bX_n)=0 \right]\\
	=& E_{\bX_n} \left[\sum_{\bd\in\mathbb{D}} \frac{\sum_{i=1}^{m}d_i(1- w_i(\bd) )}{\sum_{i=1}^{m}d_i}
	\dnm(\bd|\bX_n)\bigg{|}\dnm(\bd=\bzero|\bX_n)=0 \right]\\
	=& E_{\bX_n} \left[\sum_{\bd\in\mathbb{D}} \frac{\sum_{i=1}^{m}d_i(1- w_i(\bd) )}{\sum_{i=1}^{m}d_i} 
	I\left( \sum_{i=1}^{m}d_i>0 \right)\dnm(\bd|\bX_n) \right] 
	\frac{1}{P_{\bX_n}\left[ \dnm(\bd=\bzero|\bX_n)=0\right] } \\
	=&E_{\bX_n}\left[\sum_{\bd\in\mathbb{D}\setminus\left\lbrace \bzero\right\rbrace} \frac{\sum_{i=1}^{m}d_i(1- w_i(\bd) )}{\sum_{i=1}^{m}d_i} 
	\dnm(\bd|\bX_n) \right] \frac{1}{P_{\bX_n}\left[ \dnm(\bd=\bzero|\bX_n)=0\right] }.
	\end{align*}	
	From Theorem \ref{theorem:preferability}, we have $\frac{1}{n}\log\mfdr\rightarrow-J_{\min}$, that is, $\mfdr\rightarrow0$, as $n\rightarrow\infty$. Also we have 
	\begin{equation*}
	0\leq\sum_{\bd\in\mathbb{D}\setminus\left\lbrace \bzero\right\rbrace} \frac{\sum_{i=1}^{m}d_i(1- w_i(\bd) )}{\sum_{i=1}^{m}d_i} 
	\dnm(\bd|\bX_n)\leq \mfdr\leq1.
	\end{equation*}
	Therefore by the dominated convergence theorem, $E_{\bX_n}\left[\sum_{\bd\in\mathbb{D}\setminus\left\lbrace \bzero\right\rbrace} \frac{\sum_{i=1}^{m}d_i(1- w_i(\bd) )}{\sum_{i=1}^{m}d_i} 
	\dnm(\bd|\bX_n) \right]\rightarrow0$, as $n\rightarrow\infty$. From \ref{A2} we have $\bd^t\neq\bzero$ and from Theorem \ref{th:nmd_consistent} we have $E_{\bX_n}[\dnm(\bd^t|\bX_n)]\rightarrow1$. Thus $P_{\bX_n}\left[ \dnm(\bd=\bzero|\bX_n)=0\right]\rightarrow1$, as $n\rightarrow\infty$. This proves the result.
	
	Similarly it can be shown that $pBFDR\rightarrow 0$ as $n\rightarrow\infty$.
\end{proof}\\[2mm]

\begin{proof}[Proof of Theorem \ref{corollary:limit_bfnrs}] 
	The proof is similar to that of Theorem \ref{theorem:preferability}.
\end{proof}\\[2mm]

\begin{proof}[Proof of Corollary \ref{bfnr_lm}]
	Exploiting Theorem \ref{corollary:limit_bfnrs} and \ref{A2}, the theorem can be proved similarly as the proof of Corollary \ref{bfdr_lm}.
\end{proof}

\section{Proofs of results in Section \ref{sec:asymp_category_a}}
\begin{proof}[Proof of Theorem \ref{theorem:mpBFDR_alpha}]
	Theorem 3.4 of \ctn{chandra2017} shows that $mpBFDR$ is non-increasing in $\beta$. Hence, the maximum error that can be incurred is at $\beta=0$ where we 
	actually maximize $\sum_{i=1}^md_iw_{in}(\bd)$.
	Let \begin{align*}
	\hat\bd &=\argmax_{\bd\in\mathbb{D}} \sum_{i=1}^md_iw_{in}(\bd)= \argmax_{\bd\in\mathbb{D}}\left[  \sum_{i=1}^{m_1}d_iw_{in}(\bd)+ \sum_{i=m_1+1}^m d_iw_{in}(\bd) \right] 
	\end{align*}	
	Since the groups in $\{G_1,G_2,\cdots,G_{m_1} \}$ have no overlap with those in $\{G_{m_1+1},\cdots,G_m\}$, $\sum_{i=1}^{m_1}d_iw_{in}(\bd)$ and 
	$\sum_{i=m_1+1}^{m}d_iw_{in}(\bd)$ can be maximized separately.

	Let us define the following notations:
	\[
	Q_{\bd}^{m_1}=Q_{\bd}\cap \{1,2,\cdots,m_1\},~	Q_{\bd}^{m_1c}=\{1,2,\cdots,m_1\}\setminus Q_{\bd}^{m_1}.
	\]
	
	Now,
	\begin{align*}
	&\sum_{i=1}^{m_1} d_iw_{in}(\bd) - \sum_{i=1}^{m_1} d_i^tw_{in}(\bd^t)\\
	=& \left[ \sum_{i\in Q_{\bd}^{m_1}} d_iw_{in}(\bd) - \sum_{i\in Q_{\bd}^{m_1}} d_i^tw_{in}(\bd^t) \right] + \left[ \sum_{i\in Q_{\bd}^{m_1c}} d_iw_{in}(\bd) 
	- \sum_{i\in Q_{\bd}^{m_1c}} d_i^tw_{in}(\bd^t) \right]\\
	=& \sum_{i\in Q_{\bd}^{m_1c}} d_iw_{in}(\bd) - \sum_{i\in Q_{\bd}^{m_1c}} d_i^tw_{in}(\bd^t), 
	\end{align*}
	since for any $\bd$, $\sum_{i\in Q_{\bd}^{m_1}} d_iw_{in}(\bd)=\sum_{i\in Q_{\bd}^{m_1}} d_i^tw_{in}(\bd^t)$ by definition of $Q_{\bd}^{m_1}$.
	
	Note that $\sum_{i\in Q_{\bd}^{m_1c}} d_i^tw_{in}(\bd^t)$ can not be zero as it contradicts \ref{ass_groups} that 
	``$G_1,G_2,\cdots,G_{m_1}$ have at least one false null hypothesis." From (\ref{eq:shalizi1}) and (\ref{eq:shalizi2}), we have
	\begin{align*}
	\sum_{i\in Q_{\bd}^{m_1c}} d_iw_{in}(\bd) &\rightarrow 0 \mbox{ for all $\bd\neq\bd^t$, and}\\
	\sum_{i\in Q_{\bd}^{m_1c}} d_i^tw_{in}(\bd^t) &\rightarrow \sum_{i\in Q_{\bd}^{m_1c}} d_i^t >0.
	\end{align*} 
	Hence, for large enough $n$, for $\bd\neq\bd^t$, 
	\begin{align*}
	\sum_{i=1}^md_iw_{in}(\bd) - \sum_{i=1}^md_i^tw_{in}(\bd^t)<0.
	\end{align*}
	In other words, $\bd^t$ (or $\bd$ such that $d_i=d_i^t$ for all $i=1,\cdots,m_1$) maximizes $\sum_{i=1}^{m_1} d_iw_{in}(\bd)$ when $n$ is large enough. 
	
	Let us now consider the term $\sum_{i=m_1+1}^m d_iw_{in}(\bd).$ Note that $\sum_{i=m_1+1}^m d_i^tw_{in}(\bd^t)= 0$ by \ref{ass_groups}. 
	For any finite $n$, $\sum_{i=m_1+1}^m d_iw_{in}(\bd)$ is maximized for some decision configuration $\tilde\bd$ where $\tilde d_i=1$ for at least one 
	$i\in \{m_1+1,\cdots,m\}$. 
	In that case, $\hat{\bd}^t=(d^t_1,\ldots,d^t_{m_1},\tilde d_{m_1+1},\tilde d_{m_1+2},\ldots,\tilde d_m)$, so that
	\begin{equation*}
	\lim_{n\rightarrow\infty} \frac{\sum_{i=1}^m \hat d_i(1-w_{in}(\hat \bd))}{\sum_{i=1}^m\hat d_i}\geq \frac{1}{\sum_{i=1}^md_i^t+1},
	\end{equation*}	
	almost surely, for all data sequences.
	Boundedness of $\frac{\sum_{i=1}^m d_i(1-w_{in}(\bd))}{\sum_{i=1}^md_i}$ for all $\bd$ and $\bX_n$ ensures uniform integrability, which, in conjunction with
	the simple observation that for $\beta=0$, $P\left(\dnm(\bd=\bzero|\bX_n)=0\right)=1$ 
	for all $n\geq 1$, guarantees
	that under \ref{ass_groups} it is possible to incur $mpBFDR\geq \frac{1}{\sum_{i=1}^md_i^t+1}$ asymptotically.
	
	Now, if $G_{m_1+1},\cdots,G_m$'s are all disjoint, each consisting of only one true null hypothesis, 
	then $\sum_{i=m_1+1}^m d_iw_{in} (\bd)$ will be maximized by $\tilde\bd$ where $\tilde d_i=1$ for all $i\in \{m_1+1,\cdots,m\}$. 
	Since $d^t_i$; $i=1,\ldots,m_1$ maximizes $\sum_{i=1}^{m_1} d_iw_{in}(\bd)$ for large $n$, it follows that $\hat\bd=(d^t_1,\ldots,d^t_{m_1},1,1,\ldots,1)$
	is the maximizer of $\sum_{i=1}^m d_iw_{in}(\bd)$ for large $n$.
	In this case, almost surely for all data sequences,
	\begin{equation}
	\lim_{n\rightarrow\infty} \frac{\sum_{i=1}^m \hat d_i(1-w_{in}(\hat \bd))}{\sum_{i=1}^m\hat d_i} = \frac{m-m_1}{\sum_{i=1}^md_i^t+m-m_1}.
	\label{eq:upper_mpBFDR2}
	\end{equation}
	In this case, the maximum $mpBFDR$ that can be incurred is at $\beta=0$, and is given by  
	\begin{equation*}
	\lim_{n\rightarrow\infty} mpBFDR_{\beta=0}=\frac{m-m_1}{\sum_{i=1}^md_i^t+m-m_1}.
	\end{equation*}
	This is also the maximum $mpBFDR$ that can be incurred among all possible configurations of $G_{m_1+1},\cdots,G_m$. Hence, for any arbitrary configuration of groups, 
	the maximum $mpBFDR$ that can be incurred lies in the interval $\left( \frac{1}{\sum_{i=1}^md_i^t+1}, \frac{m-m_1}{\sum_{i=1}^md_i^t+m-m_1} \right)$ asymptotically.	
\end{proof}\\[2mm]

\begin{proof}[Proof of Theorem \ref{corollary:beta_n0}] Let $\epsilon<E-\alpha$. Then from (\ref{eq:lim_mpbfdr}), there exists $n(\epsilon)$ such that for all $n>n(\epsilon)$, 
	$mpBFDR_{\beta=0}> E-\epsilon>\alpha$. \ctn{chandra2017} have shown that $mpBFDR$ is continuous and decreasing in $\beta$. Hence, for all $n>n(\epsilon)$, 
	there exists $\beta_n\in(0,1)$ such that $mpBFDR=\alpha$. 	
	
	Now, if possible let $\lim\inf_{n\rightarrow\infty} \beta_n>0$. Then from Theorem \ref{theorem:mpBFDR_form1} we see that $mpBFDR$ decays to $0$ exponentially fast, 
	which contradicts the current situation that $mpBFDR=\alpha$ for $n>n(\epsilon)$. Hence, $\lim_{n\rightarrow\infty} \beta_n=0$.	
\end{proof}\\[2mm]

\begin{proof}[Proof of Theorem \ref{theorem:pBFDR_alpha2}]
	Theorems 3.1 and 3.4 of \ctn{chandra2017} together state that $mpBFDR$ is continuous and non-increasing in $\beta$. It is to be noted that there is no assumption or restriction on the 
	configurations of $G_i$'s. 
	Hence it is easily seen that $pBFDR$ is also continuous and non-increasing in $\beta$.
	
	Let $\hat\bd$ be the optimal decision configuration with respect to the additive loss function. Note that for $\beta=0$, $\hat d_i=1$ for all $i$. In that case,
	\begin{equation*}
	\lim_{n\rightarrow\infty}\frac{\sum_{i=1}^m \hat d_i(1-v_{in})}{\sum_{i=1}^m\hat d_i}=\frac{m_0}{m}.
	\end{equation*}
	Therefore, 
	it is possible to incur error arbitrarily close to $m_0/m$ for large enough sample size. Hence, the remaining part of the proof follows 
	in the same lines as the arguments in the proof of Theorem \ref{corollary:beta_n0}.	
\end{proof}\\[2mm]

\begin{proof}[Proof of Theorem \ref{theorem:mpBFDR_alpha2}]
	Take $\epsilon<\frac{m_0}{m}-\alpha$. Since for any multiple testing method, $mpBFDR_{\beta}>pBFDR_{\beta}$, and since $\underset{n\rightarrow\infty}{\lim}~pBFDR_{\beta=0}=\frac{m_0}{m}$ 
	by the proof of
	Theorem \ref{theorem:pBFDR_alpha2}, it follows that
	there exists $n_0(\epsilon)$ such that for all $n> n_0(\epsilon)$,
	\begin{equation*}
	mpBFDR_{\beta=0}> \frac{m_0}{m}-\epsilon>\alpha.
	\end{equation*}
	Since $mpBFDR$ is continuous and non-increasing in $\beta$, for 
	for $n>n_0(\epsilon)$, there exists a sequence $\beta_n\in[0,1]$ such that 
	\begin{equation}
	mpBFDR_{\beta_n}=\alpha.
	\end{equation}
	If possible, let $\lim\inf_{n\rightarrow\infty} \beta_n>0$. This, however, contradicts Theorem \ref{theorem:mpBFDR_form1} which asserts that $mpBFDR$ decays to $0$ exponentially fast.
	Hence, $\lim_{n\rightarrow\infty} \beta_n=0$.
\end{proof}\\[2mm]

\begin{proof}[Proof of Theorem \ref{theorem:mpBFNR_alpha}]
	From Theorem \ref{corollary:beta_n0} we have that for any feasible choice of $\alpha$, there exists a sequence $\{\beta_n\}$ such that $\lim_{n\rightarrow\infty} mpBFDR_{\beta_n}=\alpha$. Now, for the sequence $\{\beta_n\}$, let $\hat{\bd_n}$ be the optimal decision configuration for sample size $n$, that is, $\dnm \left(\hat{\bd_n}|\bX_n\right)=1$ for sufficiently large $n$. 
	Following the proof of Theorem \ref{theorem:mpBFDR_alpha} and \ref{corollary:beta_n0} we see that $\hat d_{in}=d_i^t$ for $i=1,\cdots,m_1$ and $\sum_{i=m_1+1}^m \hat d_{in}>0$.
	Now recall from (\ref{eq:shalizi1_muller}) that for any arbitrary $\epsilon>0$, there exists $n(\epsilon)$ such that for all $n>n(\epsilon)$, 
	$v_{in} < \exp\left[-n\left(J\left(H_{1i}\right)-\epsilon\right)\right]$ if $d^t_i=0$. Therefore, 
	\begin{align*}
	&\frac{\sum_{i=1}^m(1-\hat d_{in})v_{in}} {\sum_{i=1}^m(1-\hat d_{in})} \leq \frac{\sum_{i=1}^m(1- d_i^t)v_{in} } {\sum_{i=1}^m(1-\hat d_{in})} 
	< e^{n\epsilon}\times \frac{\sum_{i=1}^m(1- d_i^t)e^{-nJ(H_{1i})} } {\sum_{i=1}^m(1-\hat d_{in})}\\
	\Rightarrow~&\fnrx< e^{n\epsilon}\times \frac{\sum_{i=1}^m(1- d_i^t)e^{-nJ(H_{1i})} } {\sum_{i=1}^m(1-\hat d_{in})}\\
	\Rightarrow~&\frac{1}{n}\log\left(\fnrx \right) <\epsilon+\frac{1}{n} \log\left[\sum_{i=1}^m(1- d_i^t)e^{-nJ(H_{1i})}\right]
	-\frac{1}{n}\log\left[ \sum_{i=1}^m(1-\hat d_{in}) \right].\\ 
	\end{align*}
	Note that \begin{align*}
	&\lim_{n\rightarrow\infty}\frac{1}{n}\log\left[ \sum_{i=1}^m(1-\hat d_{in})\right]=0~\mbox{as $m$ is finite, and}\\
	&\lim_{n\rightarrow\infty} \frac{1}{n} \log\left[\sum_{i=1}^m(1- d_i^t)e^{-nJ(H_{1i})}\right] = -\tilde H_{\min} \text{ from L'H\^opital's rule.}
	\end{align*}
	As $\epsilon$ is any arbitrary positive quantity we have
	\begin{equation*}
	\limsup_{n\rightarrow\infty} \frac{1}{n}\log\left(\fnrx \right)\leq -\tilde H_{\min}.
	\end{equation*}
\end{proof}

\section{Additional results to Section \ref{sec:ar1} and proofs}
\label{subsec:A2}
\begin{proof}[Proof of Theorem \ref{theorem:ar1}]
	The proof of this theorem is complete if \ref{shalizi1}-\ref{shalizi7} are verified for the model \eqref{eq:modeled_ar1}. We do this through the following lemmas and theorems stated and proved in this section.
\end{proof}

\begin{lemma}
	\label{th:KL_ar}
	Under the model assumption \ref{ar1}-\ref{ar2}, the KL-divergence rate $h(\btheta)$ defined in \eqref{eq:KL_rate} exists and is given by
	\begin{multline}
	h(\theta)=\log\left(\frac{\sigma}{\sigma_0}\right)+\left(\frac{1}{2\sigma^2}-\frac{1}{2\sigma^2_0}\right)\left(\frac{\sigma^2_0}{1-\rho^2_0}
	+\frac{\bbeta^\prime_0\bSigma_z\bbeta_0}{1-\rho^2_0}\right)\\
	+\left(\frac{\rho^2}{2\sigma^2}-\frac{\rho^2_0}{2\sigma^2_0}\right)\left(\frac{\sigma^2_0}{1-\rho^2_0}+\frac{\bbeta^\prime_0\bSigma_z\bbeta_0}{1-\rho^2_0}\right)
	+\frac{1}{2\sigma^2}\bbeta^\prime\bSigma_z\bbeta-\frac{1}{2\sigma^2_0}\bbeta^\prime_0\bSigma_z\bbeta_0\\
	-\left(\frac{\rho}{\sigma^2}-\frac{\rho_0}{\sigma^2_0}\right)\left(\frac{\rho_0\sigma^2_0}{1-\rho^2_0}+\frac{\rho_0\bbeta^\prime_0\bSigma_z\bbeta_0}{1-\rho^2_0}\right)
	-\left(\frac{\bbeta}{\sigma^2}-\frac{\bbeta_0}{\sigma^2_0}\right)^\prime\bSigma_z\bbeta_0.\label{eq:htheta}
	\end{multline}
\end{lemma}
\begin{proof}
	It is easy to see that under the true model $P$,
	\begin{align}
	E(x_t)&=\sum_{k=1}^t\rho^{t-k}_0\bz^\prime_k\bbeta_0;
	\label{eq:mean_true}\\
	E(x_{t+h}x_t)&\sim\frac{\sigma^2_0\rho^h_0}{1-\rho^2_0}+E(x_{t+h})E(x_t);~h\geq 0,\nonumber
	\end{align}
	where for any two sequences $\{a_t\}_{t=1}^{\infty}$ and $\{b_t\}_{t=1}^{\infty}$, $a_t\sim b_t$ stands for $a_t/b_t\rightarrow 1$ as $t\rightarrow\infty$.
	Hence, 
	\begin{equation}
	E(x^2_t)\sim \frac{\sigma^2_0}{1-\rho^2_0}+\left(\sum_{k=1}^t\rho^{t-k}_0\bz^\prime_k\bbeta_0\right)^2.
	\label{eq:lim_mean_sq}
	\end{equation}

	Now let 
	\begin{equation*}
	\varrho_t=\sum_{k=1}^t\rho^{t-k}_0\bz^\prime_k\bbeta_0 
	\end{equation*}
	and for $t>t_0$, 
	\begin{equation*}
	\tilde\varrho_t=\sum_{k=t-t_0}^t\rho^{t-k}_0\bz^\prime_k\bbeta_0, 
	\end{equation*}
	where, for any $\varepsilon>0$, $t_0$ is so large that 
	\begin{equation}
	\frac{C\rho^{t_0+1}_0}{(1-\rho^{t_0})}\leq\varepsilon.
	\label{eq:small1}
	\end{equation}
	It follows, using \ref{ar2} and \eqref{eq:small1}, that for $t>t_0$,
	\begin{equation}
	\left|\varrho_t-\tilde\varrho_t\right|\leq \sum_{k=1}^{t-t_0-1}\rho^{t-k}_0\left|\bz^\prime_k\bbeta_0\right|\leq\frac{C\rho^{t_0+1}_0(1-\rho^{t-t_0+1}_0)}{1-\rho_0}\leq\varepsilon.
	\label{eq:diff_inequality1}
	\end{equation}
	Hence, for $t>t_0$,
	\begin{equation}
	\tilde\varrho_t-\varepsilon\leq\varrho_t\leq\tilde\varrho_t+\varepsilon.
	\label{eq:diff_inequality2}
	\end{equation}
	Now,
	\begin{align*}
	\frac{\sum_{t=1}^n\tilde\varrho_t}{n}&=\rho^{t_0}_0\left(\frac{\sum_{t=1}^n\bz_t}{n}\right)^\prime\bbeta_0+\rho^{t_0-1}_0\left(\frac{\sum_{t=2}^n\bz_t}{n}\right)^\prime\bbeta_0
	+\rho^{t_0-2}_0\left(\frac{\sum_{t=3}^n\bz_t}{n}\right)^\prime\bbeta_0+\cdots\notag\\
	&\qquad\qquad\cdots+\rho_0\left(\frac{\sum_{t=t_0}^n\bz_t}{n}\right)^\prime\bbeta_0+\left(\frac{\sum_{t=t_0+1}^n\bz_t}{n}\right)^\prime\bbeta_0\notag\\
	&\rightarrow 0,~\mbox{as}~n\rightarrow\infty,~\mbox{by virtue of \ref{ar1}}.
	\end{align*}
	Similarly, it is easily seen, using \ref{ar1}, that
	\begin{equation*}
	\frac{\sum_{t=1}^n\tilde\varrho^2_t}{n}\rightarrow \left(\frac{1-\rho^{2(2t_0+1)}}{1-\rho^2_0}\right)\bbeta^\prime_0\bSigma_z\bbeta_0,~\mbox{as}~n\rightarrow\infty.
	\end{equation*}
	
	Since (\ref{eq:diff_inequality1}) implies that for $t>t_0$, $\tilde \varrho^2_t+\varepsilon^2-2\varepsilon\tilde \varrho_t\leq \varrho^2_t
	\leq\tilde \varrho^2_t+\varepsilon^2+2\varepsilon\tilde \varrho_t$,
	it follows that
	\begin{equation*}
	\underset{n\rightarrow\infty}{\lim}~\frac{\sum_{t=1}^n\varrho^2_t}{n}=\underset{n\rightarrow\infty}{\lim}~\frac{\sum_{t=1}^n\tilde\varrho^2_t}{n}+\varepsilon^2
	=\left(\frac{1-\rho^{2(2t_0+1)}}{1-\rho^2_0}\right)\bbeta^\prime_0\bSigma_z\bbeta_0+\varepsilon^2,
	\end{equation*}
	and since $\epsilon>0$ is arbitrary, it follows that
	\begin{equation}
	\underset{n\rightarrow\infty}{\lim} \frac{\sum_{t=1}^n\varrho^2_t}{n}
	=\frac{\bbeta^\prime_0\bSigma_z\bbeta_0}{1-\rho^2_0}.
	\label{eq:varrho_sq_limit2}
	\end{equation}
	
	Hence, it also follows from (\ref{eq:mean_true}), (\ref{eq:lim_mean_sq}), \ref{ar1} and (\ref{eq:varrho_sq_limit2}), that 
	
	\begin{equation*}
	\frac{\sum_{t=1}^nE(x^2_t)}{n}\rightarrow \frac{\sigma^2_0}{1-\rho^2_0}+\frac{\bbeta^\prime_0\bSigma_z\bbeta_0}{1-\rho^2_0},~\mbox{as}~n\rightarrow\infty.
	\end{equation*}
	and
	\begin{equation}
	\frac{\sum_{t=1}^nE(x^2_{t-1})}{n}\rightarrow \frac{\sigma^2_0}{1-\rho^2_0}+\frac{\bbeta^\prime_0\bSigma_z\bbeta_0}{1-\rho^2_0},~\mbox{as}~n\rightarrow\infty.
	\label{eq:lim3}
	\end{equation}
	Now note that 
	\begin{equation}
	x_tx_{t-1}=\rho_0x^2_{t-1}+\bz^\prime_t\bbeta_0x_{t-1}+\epsilon_tx_{t-1}.
	\label{eq:cov_lag1}
	\end{equation}
	Using (\ref{eq:ass2}), (\ref{eq:diff_inequality2}) and arbitrariness of $\varepsilon>0$ it is again easy to see that
	\begin{equation}
	\frac{\sum_{t=1}^n\bz^\prime_t\bbeta_0E(x_{t-1})}{n}\rightarrow 0,~\mbox{as}~n\rightarrow\infty.
	\label{eq:lim4}
	\end{equation}
	Also, since for $t=1,2,\ldots,$ $E(\epsilon_tx_{t-1})=E(\epsilon_t)E(x_{t-1})$ by independence, and since $E(\epsilon_t)=0$ for $t=1,2,\ldots$, it holds that
	\begin{equation}
	\frac{\sum_{t=1}^nE\left(\epsilon_tx_{t-1}\right)}{n}= 0,~\mbox{for all}~n=1,2,\ldots.
	\label{eq:lim5}
	\end{equation}
	Combining \eqref{eq:lim3}-\eqref{eq:lim5} 
	we obtain
	\begin{equation*}
	\frac{\sum_{t=1}^nE\left(x_tx_{t-1}\right)}{n}\rightarrow \frac{\rho_0\sigma^2_0}{1-\rho^2_0}+\frac{\rho_0\bbeta^\prime_0\bSigma_z\bbeta_0}{1-\rho^2_0}.
	\end{equation*}
	Also \ref{ar1} along with (\ref{eq:diff_inequality2}) and arbitrariness of $\varepsilon>0$ yields
	\begin{align*}
	\frac{\sum_{t=1}^n\bz_tE(x_t)}{n}\rightarrow \bSigma_z\bbeta_0,~\mbox{as}~n\rightarrow\infty;\\
	\frac{\sum_{t=1}^n\bz_tE(x_{t-1})}{n}\rightarrow \bzero,~\mbox{as}~n\rightarrow\infty.
	\end{align*}
	Using assumptions \ref{ar1} and \ref{ar2} and the above results, it follows that
	\begin{multline*}
	h(\btheta)=\underset{n\rightarrow\infty}{\lim}~\frac{1}{n}E\left[-\log R_n(\btheta)\right]=\log\left(\frac{\sigma}{\sigma_0}\right)+\left(\frac{1}{2\sigma^2}-\frac{1}{2\sigma^2_0}\right)\left(\frac{\sigma^2_0}{1-\rho^2_0}
	+\frac{\bbeta^\prime_0\bSigma_z\bbeta_0}{1-\rho^2_0}\right)\\
	+\left(\frac{\rho^2}{2\sigma^2}-\frac{\rho^2_0}{2\sigma^2_0}\right)\left(\frac{\sigma^2_0}{1-\rho^2_0}+\frac{\bbeta^\prime_0\bSigma_z\bbeta_0}{1-\rho^2_0}\right)
	+\frac{1}{2\sigma^2}\bbeta^\prime\bSigma_z\bbeta-
	\frac{1}{2\sigma^2_0}\bbeta^\prime_0\bSigma_z\bbeta_0\\
	-\left(\frac{\rho}{\sigma^2}-\frac{\rho_0}{\sigma^2_0}\right)\left(\frac{\rho_0\sigma^2_0}{1-\rho^2_0}+\frac{\rho_0\bbeta^\prime_0\bSigma_z\bbeta_0}{1-\rho^2_0}\right)
	-\left(\frac{\bbeta}{\sigma^2}-\frac{\bbeta_0}{\sigma^2_0}\right)^\prime\bSigma_z\bbeta_0.
	\end{multline*}
	In other words, \ref{s2} holds, with $h(\btheta)$ given by equation \eqref{eq:htheta}. 
\end{proof}

\begin{theorem}
	\label{th:eqpartition}
	For each $\btheta \in \bTheta$, the generalized or relative asymptotic equipartition property holds, and so
	\begin{equation*}
	\lim_{n\rightarrow\infty} \frac{1}{n}\log R_n(\btheta) = -h(\btheta).
	\end{equation*}
	The convergence is uniform over any compact subset of $\bTheta$.
\end{theorem}
\label{subsec:A3}
\begin{proof}
	Note that 
	\begin{equation*}
	x_t=\sum_{k=1}^t\rho^{t-k}_0\bz^\prime_k\bbeta_0+\sum_{k=1}^t\rho^{t-k}_0\epsilon_k,
	\end{equation*}
	where $\tilde\epsilon_t=\sum_{k=1}^t\rho^{t-k}_0\epsilon_k$ is an asymptotically stationary Gaussian process with mean zero and covariance
	\begin{equation*}
	cov(\tilde\epsilon_{t+h},\tilde\epsilon_t)\sim\frac{\sigma^2_0\rho^h_0}{1-\rho^2_0},~\mbox{where}~h\geq 0.
	\end{equation*}
	Then 
	\begin{equation}
	\frac{\sum_{t=1}^nx^2_t}{n}=\frac{\sum_{t=1}^n\varrho^2_t}{n}+\frac{\sum_{t=1}^n\tilde\epsilon^2_t}{n}+
	\frac{2\sum_{t=1}^n\tilde\epsilon_t\varrho_t}{n}.
	\label{eq:xt_sq_average}
	\end{equation}
	By (\ref{eq:varrho_sq_limit2}), the first term of the right hand side of (\ref{eq:xt_sq_average}) converges to $\frac{\bbeta^\prime_0\bSigma_z\bbeta_0}{1-\rho^2_0}$, 
	as $n\rightarrow\infty$, and since $\tilde\epsilon_t$; $t=1,2,\ldots$
	is also an irreducible and aperiodic Markov chain, by the ergodic theorem it follows that the second term of (\ref{eq:xt_sq_average}) converges to $\sigma^2_0/(1-\rho^2_0)$
	almost surely, as $n\rightarrow\infty$.
	Also observe that $\varrho_t$; $t=1,2,\ldots$, is also a sample path of an irreducible and aperiodic stationary Markov chain, with univariate stationary distribution 
	having mean $0$ and variance
	$\bbeta^\prime_0\bSigma_z\bbeta_0\times\underset{t\rightarrow\infty}{\lim}~\sum_{k=1}^t\rho^{2(t-k)}_0=\frac{\bbeta^\prime_0\bSigma_z\bbeta_0}{1-\rho^2_0}$. 
	Since for each $t$, $\tilde\epsilon_t$ and $\varrho_t$ are independent, $\tilde\epsilon_t\varrho_t$; $t=1,2,\ldots$, is also an irreducible and aperiodic 
	Markov chain having a stationary distribution with mean 0 and variance $\frac{\sigma^2_0\bbeta^\prime_0\bSigma_z\bbeta_0}{(1-\rho^2_0)^2}$.
	Hence, by the ergodic theorem, the third term of (\ref{eq:xt_sq_average}) converges to zero, almost surely, as $n\rightarrow\infty$. It follows that
	\begin{equation}
	\frac{\sum_{t=1}^nx^2_t}{n}\rightarrow\frac{\sigma^2_0}{1-\rho^2_0}+\frac{\bbeta^\prime_0\bSigma_z\bbeta_0}{1-\rho^2_0},
	\label{eq:xt_sq_average2}
	\end{equation}
	and similarly,
	\begin{equation}
	\frac{\sum_{t=1}^nx^2_{t-1}}{n}\rightarrow\frac{\sigma^2_0}{1-\rho^2_0}+\frac{\bbeta^\prime_0\bSigma_z\bbeta_0}{1-\rho^2_0}.
	\label{eq:xt_sq_average3}
	\end{equation}

	Now, since $x_t=\varrho_t+\tilde\epsilon_t$, it follows using \ref{ar1} and (\ref{eq:diff_inequality2}) that
	\begin{equation}
	\underset{n\rightarrow\infty}{\lim}~\frac{\sum_{t=1}^n\bz_tx_t}{n}=\left(\underset{n\rightarrow\infty}{\lim}~\frac{\sum_{t=1}^n\bz_t\bz^\prime_t}{n}\right)\bbeta_0
	+\underset{n\rightarrow\infty}{\lim}~\frac{\sum_{t=1}^n\bz_t\tilde\epsilon_t}{n}.
	\label{eq:z_tx_t_average}
	\end{equation}
	By \ref{ar1}, the first term on the right hand side of (\ref{eq:z_tx_t_average}) is $\bSigma_z\bbeta_0$. For the second term, note that it follows from \ref{ar1} that 
	$\bz_t\tilde\epsilon_t$; $t=1,2,\ldots$,
	is sample path of an irreducible and aperiodic Markov chain with a stationary distribution having zero mean. Hence, by the ergodic theorem, it follows that
	the second term of (\ref{eq:z_tx_t_average}) is $\bzero$, almost surely. In other words, almost surely,
	\begin{equation}
	\frac{\sum_{t=1}^n\bz_tx_t}{n}\rightarrow \bSigma_z\bbeta_0,~\mbox{as}~n\rightarrow\infty,
	\label{eq:z_tx_t_average2}
	\end{equation}
	and similar arguments show that, almost surely,
	\begin{equation}
	\frac{\sum_{t=1}^n\bz_tx_{t-1}}{n}\rightarrow \bzero,~\mbox{as}~n\rightarrow\infty.
	\label{eq:z_tx_t_average3}
	\end{equation}
	
	We now calculate the limit of $\sum_{t=1}^nx_tx_{t-1}/n$, as $n\rightarrow\infty$. By (\ref{eq:cov_lag1}), 
	\begin{equation}
	\underset{n\rightarrow\infty}{\lim}~\frac{\sum_{t=1}^nx_tx_{t-1}}{n}=\underset{n\rightarrow\infty}{\lim}~\frac{\rho_0\sum_{t=1}^nx^2_{t-1}}{n}
	+\underset{n\rightarrow\infty}{\lim}~\frac{\bbeta^\prime_0\sum_{t=1}^n\bz_tx_{t-1}}{n}
	+\underset{n\rightarrow\infty}{\lim}~\frac{\sum_{t=1}^n\epsilon_tx_{t-1}}{n}.
	\label{eq:cov_lag1_limit}
	\end{equation}
	By (\ref{eq:xt_sq_average3}), the first term on the right hand side of (\ref{eq:cov_lag1_limit}) is given, almost surely, by 
	$\frac{\rho_0\sigma^2_0}{1-\rho^2_0}+\frac{\rho_0\bbeta^\prime_0\bSigma_z\bbeta_0}{1-\rho^2_0}$, and the second term is almost surely zero due to (\ref{eq:z_tx_t_average3}).
	For the third term, note that $\epsilon_tx_{t-1}=\epsilon_t\varrho_{t-1}+\epsilon_t\tilde\epsilon_{t-1}$. Both $\epsilon_t\varrho_{t-1}$; $t=1,2,\ldots$
	and $\epsilon_t\tilde\epsilon_{t-1}$; $t=1,2,\ldots$, are sample paths of irreducible and aperiodic Markov chains having stationary distributions with mean zero.
	Hence, by the ergodic theorem, the third term of (\ref{eq:cov_lag1_limit}) is zero, almost surely. That is,
	\begin{equation}
	\underset{n\rightarrow\infty}{\lim}~\frac{\sum_{t=1}^nx_tx_{t-1}}{n}=\frac{\rho_0\sigma^2_0}{1-\rho^2_0}+\frac{\rho_0\bbeta^\prime_0\bSigma_z\bbeta_0}{1-\rho^2_0}.
	\label{eq:cov_lag1_limit2}
	\end{equation}

	The limits (\ref{eq:xt_sq_average2}), (\ref{eq:xt_sq_average3}), (\ref{eq:z_tx_t_average2}), (\ref{eq:z_tx_t_average3}), (\ref{eq:cov_lag1_limit2})
	applied to $\log R_n(\btheta)$ given by  Theorem \ref{th:eqpartition}, shows that $\frac{\log R_n(\theta)}{n}$ converges to $-h(\theta)$ almost surely as $n\rightarrow\infty$.
	In other words, \ref{s3} holds.
	
	Now $\frac{1}{n}\log R_n(\btheta)$ has continuous partial derivatives implying that $\frac{\partial}{\partial\btheta} \left[\frac{1}{n}\log R_n(\btheta)\right] $ is bounded in any compact set. Hence $\frac{1}{n}\log R_n(\btheta)$ is Lipschitz continuous and hence stochastic equicontinuous in $\btheta$. Thus by applying the stochastic Ascoli theorem we have that the convergence is uniform over $\btheta$ in that compact set (for details about stochastic equicontinuity, see, for example, \ctn{billingsley2013}).
\end{proof}
The meaning of Theorem \ref{th:eqpartition} is that, relative to the true distribution, the
likelihood of each $\btheta$ goes to zero exponentially, the rate being the Kullback-Leibler divergence rate. Roughly speaking, an integral of exponentially-shrinking quantities will tend to be dominated by the integrand with the slowest rate of decay. Lemma \ref{th:KL_ar} and Theorem \ref{th:eqpartition} imply that \ref{shalizi1}-\ref{s3} hold. For any $\btheta\in\bTheta$, $h(\btheta)$ is finite, which implies that $\ref{s4}$ also holds. As regards \ref{s5}, we can always make \eqref{eq:S5_1} to hold by considering $\mathcal G_n$s as credible regions of the prior distribution and these can be chosen increasing compact sets without loss of generality. Since $h(\cdot)$ is continuous in $\btheta$ the second and third parts of $\ref{s5}$ will also hold.

Note that the maximizer of $R_n(\btheta)$ is the maximum likelihood estimator (mle) of $\btheta$. Let $\hat{\btheta}_n=\underset{\btheta\in\mathcal G_n}{\sup}~ R_n(\btheta)$. Then 
\begin{equation}
\frac{1}{n}\log\int_{\mathcal G_n}R_n(\btheta)\pi(\btheta)d\btheta
\leq \frac{1}{n}\log\left[ R_n(\hat{\btheta}_n)\pi(\mathcal G_n) \right].
\label{eq:s6_mle}
\end{equation}
If we can show that $\hat{\btheta}_n$ is a consistent estimator of $\thetao$, 
then this will validate \ref{s6}. Importantly, the conditions for mle consistency generally require $iid$ observations \ctp{Lehmann98}. In this model the data sequence $\{x_t\}_{t=1}^\infty$ have dependence structure and regular asymptotic theory will not hold. 
Hence, we provide a direct proof of consistency; below we provide the main results leading to the desired consistency result. The
equipartition property plays a crucial role in the proceeding.



\begin{theorem}
	\label{th:concave}
	The function $\frac{1}{n}\log R_n(\btheta)$ is asymptotically concave in $\btheta$.
\end{theorem}
\begin{proof}
	Note that 
	\begin{equation*}
	\sup_{\btheta\in\bTheta} \frac{1}{n}\log R_n(\btheta) = \sup_{\rho,\bbeta} \sup_{\sigma^2}\frac{1}{n}\log R_n(\btheta)= -\inf_{\rho,\bbeta}~\log\left[   \frac{1}{n} \sum_{t=1}^{n} \left(x_t -\rho x_{t-1}-\bbeta'\bz_t \right)^2\right]-\frac{1}{2}.
	\end{equation*}
	Since $\log$ is a monotonic function, minimizing $\log\left[\frac{1}{n} \sum_{t=1}^{n} \left(x_t -\rho x_{t-1}-\bbeta'\bz_t \right)^2\right]$ is equivalent to minimizing $\frac{1}{n} \sum_{t=1}^{n} \left(x_t -\rho x_{t-1}-\bbeta'\bz_t \right)^2=g_n(\rho,\bbeta),$ say. Now the Jacobian matrix $J$ of $g_n(\rho,\bbeta)$ is given by 
	\begin{equation*}
	J= \begin{bmatrix}
	\frac{1}{n}\sum x^2_{t-1} & \frac{1}{n}\sum x_{t-1}\bz'_t\\
	\frac{1}{n}\sum x_{t-1}\bz_t & \frac{1}{n}\sum \bz_t\bz'_t
	\end{bmatrix}.
	\end{equation*}
	\eqref{eq:xt_sq_average2}, \eqref{eq:z_tx_t_average3} together with the model assumptions \ref{ar1}-\ref{ar2} clearly shows that for large enough $n$, $J$ is positive-definite. Hence $g_n(\rho,\bbeta)$ is convex implying that $\frac{1}{n}\log R_n(\btheta)$ is a concave function for large $n$.
\end{proof}

The above theorem ensures that for large enough $n$, the likelihood equation have unique mle. Rest we need to ensure the strong consistency of the mle for this dependent setup.

\begin{theorem}
	\label{th:uniroot}
	Given any $\eta>0$, the log-likelihood ratio $\frac{1}{n}\log R_n(\btheta)$ has its unique root in the $\eta$-neighbourhood of $\thetao$ almost surely for large $n$.
\end{theorem}
\begin{proof}
	\ref{ar3} ensures that $\thetao$ is an interior point in $\bTheta$, implying that there exists a compact set $G\subset\bTheta$ such that $\thetao$ is an interior point of $G$ also. 
	From Theorem \ref{th:eqpartition}, for each $\btheta \in \bTheta$, we have
	\begin{equation}
	\lim_{n\rightarrow\infty} \frac{1}{n}\log R_n(\btheta) = -h(\btheta),
	\label{eq:equi}
	\end{equation}
	and the convergence in \eqref{eq:equi} is uniform over $\btheta$ in $G$. Thus,
	
	\begin{equation}
	\lim_{n\rightarrow\infty}\sup_{\btheta\in G}~\abs{ \frac{1}{n}\log R_n(\btheta)  + h(\btheta)}=0.
	\label{eq:unifG}
	\end{equation}
	For any $\eta>0$, we define 
	\[
	N_\eta(\btheta_0)=\{\btheta:\norm{\btheta_0-\btheta}<\eta\};~N'_\eta(\btheta_0)=\{\btheta:\norm{\btheta_0-\btheta}=\eta\};~\overline{N}_\eta(\thetao)=\{\btheta:\norm{\btheta_0-\btheta}\leq\eta\}.
	\]
	Note that for sufficiently small $\eta$, $\overline{N}_\eta(\thetao)\subset G$. Let $H=\underset{\btheta\in N'_\eta(\btheta_0)}{\inf}~h(\btheta)$. 
	By the properties of KL-divergence $h(\btheta)$ is minimum at $\btheta=\btheta_0$ and therefore, $H>0$. Let us fix an $\varepsilon$ such that $0<\varepsilon<H$. Then by \eqref{eq:unifG}, for large enough $n$ all $\btheta\in N'_\eta(\thetao)$,
	$\frac{1}{n}\log R_n(\btheta)  <- h(\btheta)+\varepsilon<0$. Now by definition $\frac{1}{n}\log R_n(\thetao)=0$ and thus for all $\btheta\in N'_\eta(\thetao)$
	\begin{equation}
	\frac{1}{n}\log R_n(\btheta)<\frac{1}{n}\log R_n(\thetao)
	\label{eq:boundary}
	\end{equation}
	for large enough $n$.
	Now, $\overline{N}_\eta(\thetao)$ is a compact set with $N'_\eta(\btheta_0)$ being its boundary. Since $\frac{1}{n}\log R_n(\btheta)$ is continuous in $\btheta$, it is bounded in $\overline{N}_\eta(\thetao)$. From \eqref{eq:boundary} we see that the maximum is attained at some interior point of $\overline{N}_\eta(\thetao)$ and not on the boundary. Since the supremum is attained at an interior point of $\overline{N}_\eta(\thetao)$, the supremum is also a local maximum. Now, Theorem \ref{th:concave} ensures that for large $n$ the maximizer of $\frac{1}{n}\log R_n(\btheta)$ is unique. This proves the result.
\end{proof}

Theorem \ref{th:uniroot} essentially entails the strong consistency of the mle. This also leads to the verification of \ref{s6} required for posterior consistency. We formally state it in the following lemma.

\begin{lemma}
	\label{lemma:s6_veify}
	For any proper prior distribution $\pi(\cdot)$ over the parameter space $\bTheta$, we have
	\begin{equation*}
	\limsup_{n\rightarrow\infty}\frac{1}{n}\log\int_{\mathcal G_n}R_n(\btheta)\pi(\btheta)d\btheta \leq 0.
	\end{equation*}
\end{lemma}
\begin{proof}
	From Theorems \ref{th:eqpartition} and \ref{th:uniroot} we have
	\begin{equation*}
	\lim_{n\rightarrow\infty} \frac{1}{n}\log R_n(\hat{\btheta}_n)=0,
	\end{equation*}
	and hence
	\begin{equation*}
	\limsup_{n\rightarrow\infty}\frac{1}{n}\log\int_{\mathcal G_n}R_n(\btheta)\pi(\btheta)d\btheta \leq \lim_{n\rightarrow\infty} \frac{1}{n}\left[ \log R_n(\hat{\btheta}_n) + \log\pi(\mathcal{G}_n) \right] \leq 0.
	\end{equation*}
\end{proof}

Lemma \ref{lemma:s6_veify} signifies that \ref{s6} holds. About \ref{shalizi7}, it trivially holds since $h(\cdot)$ is a continuous function.

\section{Supplementary to real data analysis}
	\begin{longtable}{|c|p{14cm}|}
		\caption{Causal SNPs for different populations}
		\label{variability_impl_mech}
		\endfirsthead
		\endhead
		\hline 
		\pbox{1cm}{Popu-\\lation}& Causal SNP\\\hline
		1& {m1, m12, m13, m114, m135, m146, m147, m236, m249, m274, m275, m276, m407, m422, m449, m537, m620, m665, m674, m680, m709, m765, m887, m894, m895, m899, m934, m951, m955, m1076, m1161, m1234, m1249, m1291, m1328, m1412, m1436, m1437, m1445, m1456, m1575, m1646, m1733, m1761, m1762, m1763, m1764, m1765, m1766, m1767, m1768, m1946, m2043, m2093, m2169, m2174, m2175, m2205, m2287, m2348, m2349, m2374, m2403, m2451, m2452, m2467, m2468, m2508, m2610, m2677, m2678, m2679, m2680, m2681, m2682, m2687, m2688, m2689, m2692, m2817, m2906, m2907, m2943, m2951, m2952, m2953, m2954, m2955, m2956, m2962, m2996, m2997, m3106, m3279, m3280, m3281, m3282, m3283, m3358, m3418, m3457, m3489, m3490, m3491, m3545, m3571, m3644, m3735, m3738, m3795, m3931, m3951, m3952, m4015, m4038, m4144, m4188, m4281, m4297, m4372, m4373, m4374, m4375, m4499, m4500, m4504, m4506, m4538, m4674, m4766, m4767, m4768, m4919, m4924, m4925, m4973, m4974, m5041, m5149, m5199, m5228, m5318, m5352, m5353, m5411, m5437, m5505, m5515, m5516, m5517, m5646, m5688, m5728, m5766, m5926, m5927, m6025, m6066, m6116, m6117, m6158, m6159, m6160, m6161, m6296, m6359, m6365, m6394, m6395, m6396, m6397, m6398, m6399, m6400, m6401, m6402, m6434, m6466, m6473, m6505, m6507, m6573, m6574, m6599, m6617, m6723, m6757, m6765, m6766, m6816, m6817, m6818, m6851, m6852, m6853, m6858, m6859, m6860, m6872, m6903, m6995, m7085, m7089, m7156, m7202, m7253, m7325, m7338, m7348}  \\ 	\hline 
		2&{m1, m147, m432, m440, m458, m589, m597, m598, m599, m600, m741, m1010, m1011, m1039, m1046, m1047, m1048, m1049, m1050, m1051, m1052, m1053, m1120, m1350, m1362, m1620, m1670, m1812, m2014, m2027, m2028, m2143, m2144, m2200, m2201, m2203, m2213, m2295, m2421, m2439, m2521, m2569, m2573, m2586, m2795, m2797, m3216, m3412, m3560, m3615, m3727, m3728, m3729, m3730, m3956, m4141, m4273, m4328, m4421, m4453, m4454, m4510, m4742, m4776, m4777, m4809, m4826, m4827, m4828, m4988, m5229, m5375, m5542, m5544, m5590, m5674, m5803, m5804, m5805, m5885, m5886, m5887, m5936, m5997, m6004, m6016, m6017, m6018, m6019, m6312, m6320, m6342, m6343, m6457, m6485, m6486, m6492, m6493, m6652, m7178, m7189, m7220, m7269, m7270}\\ \hline 
		3&{m1, m113, m162, m172, m176, m196, m198, m443, m446, m538, m678, m777, m796, m896, m917, m945, m946, m947, m1042, m1238, m1318, m1324, m1325, m1468, m1740, m1905, m1906, m2039, m2062, m2148, m2162, m2163, m2197, m2202, m2265, m2409, m2435, m2502, m2627, m2628, m2629, m2691, m2881, m2988, m3414, m3415, m3438, m3439, m3440, m3441, m3442, m3443, m3446, m3447, m3448, m3484, m3543, m3544, m3811, m3848, m3849, m3850, m3851, m4297, m4318, m4425, m4493, m4578, m5030, m5031, m5094, m5223, m5379, m5380, m5448, m5685, m5706, m5799, m5808, m5911, m5931, m6095, m6096, m6395, m6396, m6397, m6401, m6434, m6473, m6481, m6527, m6729, m7059, m7066}\\\hline
		4&{m1, m61, m71, m72, m143, m144, m181, m182, m183, m225, m294, m295, m555, m707, m799, m805, m816, m937, m1020, m1156, m1262, m1312, m1316, m1318, m1323, m1335, m1377, m1408, m1448, m1483, m1484, m1485, m1486, m1487, m1488, m1586, m1714, m1715, m1716, m1717, m1776, m1822, m1840, m1909, m1946, m1947, m1948, m1949, m2013, m2120, m2121, m2129, m2130, m2131, m2139, m2179, m2191, m2227, m2228, m2263, m2286, m2413, m2433, m2434, m2481, m2507, m2621, m2785, m2816, m2832, m2895, m2902, m2922, m2923, m2937, m2964, m2965, m3168, m3170, m3211, m3290, m3312, m3378, m3412, m3493, m3495, m3496, m3604, m3614, m3680, m3819, m3820, m3837, m3838, m3839, m3840, m3978, m4069, m4150, m4165, m4166, m4216, m4217, m4218, m4219, m4220, m4221, m4222, m4234, m4235, m4351, m4374, m4375, m4377, m4378, m4504, m4580, m4581, m4582, m4583, m4584, m4585, m4586, m4587, m4588, m4612, m4637, m4648, m4692, m4712, m4713, m4714, m4715, m4776, m4818, m4833, m4918, m5016, m5079, m5152, m5153, m5154, m5233, m5234, m5237, m5326, m5379, m5453, m5454, m5455, m5456, m5457, m5747, m5789, m5794, m5825, m5854, m5859, m5893, m5894, m5904, m5923, m5948, m6084, m6152, m6153, m6154, m6155, m6431, m6432, m6438, m6445, m6502, m6503, m6509, m6560, m6756, m6853, m6869, m6870, m6897, m6898, m6899, m6900, m6967, m7071, m7072, m7257, m7267, m7268, m7272}\\\hline
		5&m1, m96, m329, m363, m431, m904, m951, m1406, m1893, m2150, m2357, m2359, m2360, m2463, m2547, m2551, m2570, m2621, m2622, m2623, m3287, m3983, m3984, m3985, m4822, m5168, m5186, m5222, m5223, m5404, m5405, m5416, m5425, m5699, m5706, m5880, m5881, m5914, m5925, m5926, m5927, m5928, m5929, m5930, m5931, m6408, m6440, m6494, m6533, m6538, m6562, m7067, m7076, m7078, m7080, m7081, m7082, m7083, m7091, m7092, m7093, m7188, m7223, m7224, m7227, m7248, m7249\\\hline
	\end{longtable}

\label{sec:supp_realdata}
\begin{figure}
	\centering
	\includegraphics[width=1\linewidth]{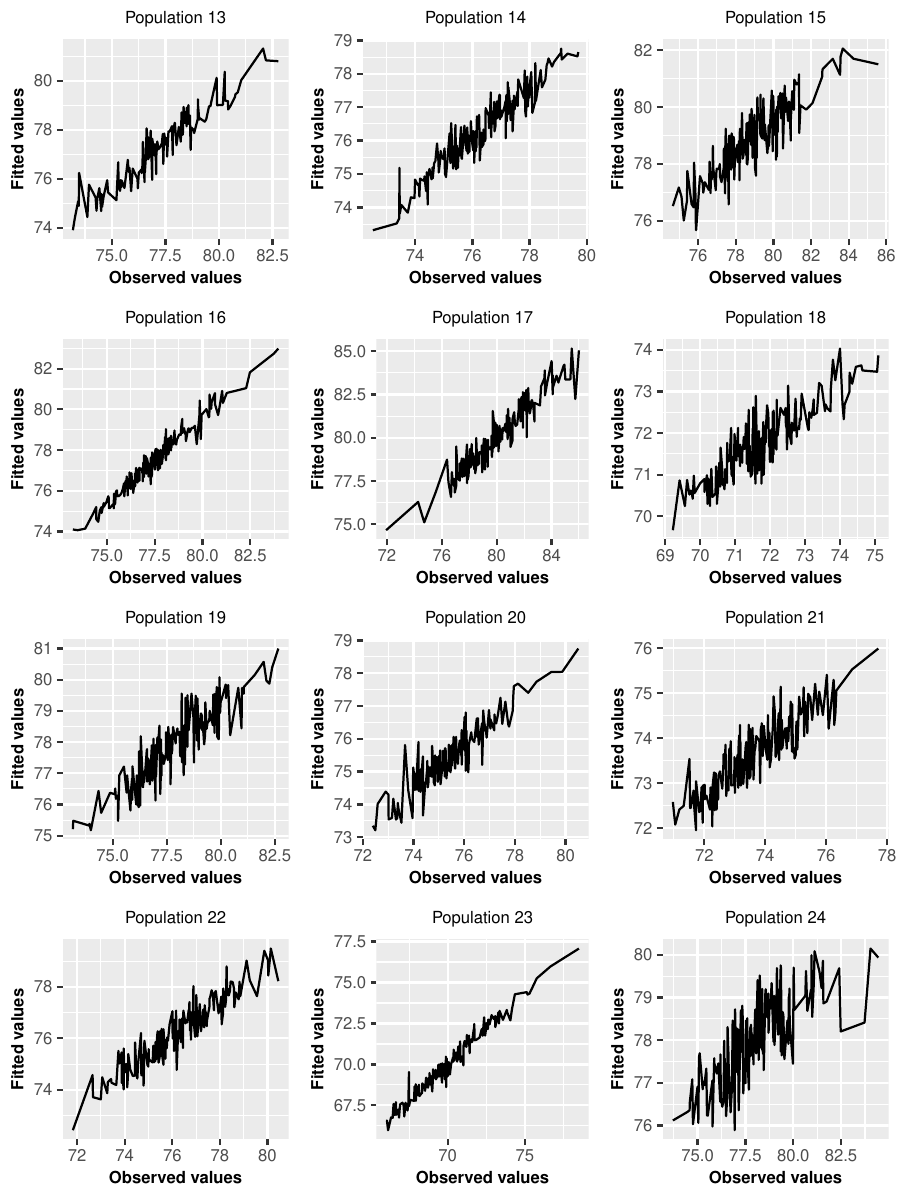}
	\caption{Fitted values versus observed values.}
	\label{fig:real_data_supp}	
\end{figure}

\bibliographystyle{natbib}
\bibliography{irmcmc}

\newpage

\end{document}